\documentclass[11pt]{article}

\usepackage{fullpage}

\usepackage{amssymb,amsmath}
\usepackage{aligned-overset}

\usepackage{amsfonts}
\usepackage{amscd}

\usepackage{subcaption}
\usepackage{graphicx}
\usepackage{epsfig}

\usepackage{esvect} %
\usepackage{enumerate}
\usepackage{enumitem}

\usepackage{dsfont}
\usepackage{mathtools}
\usepackage{bm}
\usepackage{cancel}

\usepackage{color}
\usepackage{todonotes}
\usepackage{tgpagella}

\usepackage{enumitem}

\usepackage{latexsym}
\usepackage{amsmath}
\usepackage{amsthm}
\usepackage{amsfonts}
\usepackage{amssymb}
\usepackage{bm}

\usepackage{url}
\usepackage{algorithm}
\usepackage{algorithmic}

\newcommand{\bsb}{{\bm{b}}}

\newcommand{\bse}{{\bm{e}}}

\newcommand{\bsx}{{\bm{x}}}
\newcommand{\bsy}{{\bm{y}}}
\newcommand{\bsz}{{\bm{z}}}

\newcommand{\bsdelta}{{\boldsymbol{\delta}}}

\newcommand{\bsmu}{{\boldsymbol{\mu}}}
\newcommand{\bsnu}{{\boldsymbol{\nu}}}
\newcommand{\bsxi}{{\boldsymbol{\xi}}}

\newcommand{\bsrho}{{\boldsymbol{\rho}}}
\newcommand{\bsvarrho}{{\boldsymbol{\varrho}}}

\newcommand{\rd}{{\mathrm{d}}}

\newcommand{\bbC}{{\mathbb{C}}}

\newcommand{\bbN}{{\mathbb{N}}}

\newcommand{\bbR}{{\mathbb{R}}}

\newcommand{\bbT}{{\mathbb{T}}}

\newcommand{\C}{{\mathbb{C}}} %
\newcommand{\N}{{\mathbb{N}}} %
\newcommand{\R}{{\mathbb{R}}} %
\newcommand{\Z}{{\mathbb{Z}}} %
\newcommand{\NN}{{\mathbb{N}}} %
\DeclareSymbolFont{bbold}{U}{bbold}{m}{n}
\DeclareSymbolFontAlphabet{\mathbbold}{bbold}

\newcommand{\calB}{{\mathcal{B}}}

\newcommand{\calF}{{\mathcal{F}}}
\newcommand{\calG}{{\mathcal{G}}}

\newcommand{\calN}{{\mathcal{N}}}
\newcommand{\calO}{{\mathcal{O}}}

\newcommand{\calS}{{\mathcal{S}}}

\newcommand{\calX}{{\mathcal{X}}}
\newcommand{\calY}{{\mathcal{Y}}}

\newcommand{\ii}{{\mathrm i}} %

\newcommand{\eps}{{\varepsilon}}

\newcommand{\cC}{\mathcal{C}}

\newcommand{\cS}{\mathcal{S}}

\newcommand{\cG}{\mathcal{G}}
\newcommand{\cF}{\mathcal{F}}

\newcommand{\bsnul}{{\boldsymbol{0}}}

\newcommand{\norm}[2][]{\|#2\|_{#1}}

\newcommand{\set}[2]{\{#1 : #2\}}

\DeclareMathOperator{\supp}{supp}

\newcommand{\Cpsi}{C_\psi}

\newcommand{\be}{\begin{equation}}
\newcommand{\ee}{\end{equation}}

\theoremstyle{definition}
\newtheorem{definition}{Definition}[section]
\newtheorem{remark}{Remark}[section]

\newtheorem{theorem}{Theorem}[section]
\newtheorem{lemma}[theorem]{Lemma}
\newtheorem{assumption}[theorem]{Assumption}
\newtheorem{proposition}[theorem]{Proposition}

\newcommand{\dotH}{\dot{H}}
\newcommand{\infcard}{{\aleph_0}}

\begin{document}

\title{Sparsity for Infinite-Parametric Holomorphic Functions\\
  on Gaussian Spaces\thanks{CM and JZ acknowledge the support of the National Group for Scientific
  Computing (GNCS - INDAM) through the Visiting Professors program. CM acknowledges the support of the Italian Ministry of University and Research (MUR) through the PRIN 2022 PNRR project NOTES (P2022NC97R), funded by the European Union – Next Generation EU and the PRIN 2022 project ASTICE (202292JW3F). JZ
  acknowledges support from the German Research Foundation (DFG) within the
  Priority Programme SPP 2298, \textit{Theoretical Foundations of Deep Learning}
  (project number 543965776).}}

\author{
  Carlo Marcati\thanks{Dipartimento di Matematica ``Felice Casorati'', Università di Pavia,
 27100 Pavia, Italy. \texttt{carlo.marcati@unipv.it}}
  \and
  Christoph Schwab\thanks{SAM, ETH Z\"urich, ETH Zentrum, HG G57.1, CH 8092 Z\"urich, Switzerland. \texttt{schwab@math.ethz.ch}}
  \and
  Jakob Zech\thanks{IWR, Uni Heidelberg, INF 205, 69120 Heidelberg, Germany. \texttt{jakob.zech@uni-heidelberg.de}}
} 

\maketitle
\begin{abstract}
We investigate the sparsity of Wiener polynomial chaos expansions of
holomorphic maps $\cG$ on Gaussian Hilbert spaces, as arise in the coefficient-to-solution
maps of linear, second order, divergence-form elliptic PDEs with log-Gaussian
diffusion coefficient.
Representing the Gaussian random field input as an affine-parametric
expansion, the nonlinear map becomes a countably-parametric, deterministic
holomorphic map of the coordinate sequence
$\bm{y} = (y_j)_{j\in\mathbb{N}} \in \mathbb{R}^\infty$.
We establish weighted
summability results for the Wiener-Hermite coefficient sequences of images of
affine-parametric expansions of the log-Gaussian input under $\cG$.
These results give rise to $N$-term approximation rate bounds for the full range
of input summability exponents $p\in (0,2)$. We show that these approximation
rate bounds apply to parameter-to-solution maps for elliptic diffusion PDEs with
lognormal coefficients.
\end{abstract}



%
\section{Introduction}
\label{sec:Intro}
The mathematical and numerical analysis of partial differential equations (PDEs)
with uncertain input data (coefficients, domains, source and boundary terms, etc.)
is a key theme in the area of uncertainty quantification (UQ), which has
seen significant development in the past decade.
Roughly speaking, it involves the study of
(i)
\emph{parametric regularity of data-to-solution maps} of PDEs
with input data from (possibly infinite-dimensional) spaces subject to probability
measures supported on these spaces,
and
(ii) the pushforwards of these measures %
to spaces of PDE solutions and quantities of interest.
We refer to, e.g., \cite{ChSCJG11,CoDe2015,DNSZ2023} and the references therein
for accounts of results in this direction.

Most mathematical results are based on two key ideas. First,
bases (such as wavelets or eigenbases of covariance operators) are introduced
in the input space,
resulting in parametric representations of admissible PDE inputs. Then, the
regularity  of the solutions is analyzed in terms of these parameters.
Results of this type have proved relevant recently also in the mathematical and numerical
analysis of so-called \emph{``operator learning''} approaches in Scientific Machine Learning
\cite{HSZ22_1015,NeurOp,Franco2024}.

\subsection{Existing Results}
\label{sec:ExRes}
The problems considered in this manuscript have been studied with coordinate parameter domains
that are closed and bounded (hence compact) subsets of the real line.
See, e.g. \cite{HSZ22_1015} and references there.
Fewer results have been developed for unbounded parameter domains, which arise in the
important setting of Gaussian measures on input spaces. These arise in numerous
applications in engineering and the sciences.
We mention only subsurface flow \cite{DNCirculant97} modelling with Finite Elements
and Bayesian PDE inversion subject to Gaussian or subexponential priors \cite{HKS21}.
Although
the introduction of coordinates is, in the case of Gaussian measures,
well understood \cite{LPFrame09},
corresponding sparsity results for
parametric solution families of elliptic and parabolic PDEs
with log-Gaussian random field inputs
have advanced, upon early work in \cite{HoangSchwab14,G10_109,EGSZ14_1045,BCDM2016},
for a particular, scalar linear elliptic divergence-form PDE
significantly in the past five years.
See  \cite{DNSZ2023} and the references there.
Polynomial approximability of holomorphic maps between Banach spaces has received
attention for some time (e.g., \cite{MujDinAppPrBSpcI,DinMujiAppBSpcII,DinMujiAppBSpcIII}).
The question of \emph{approximation rates in the infinite-dimensional setting}
was addressed only more recently, particularly in the context of numerical analysis
of operator surrogates (e.g., \cite{HSZ22_1015} and \cite{NeurOp}).

The corresponding deterministic
parametric regularity analysis involves \emph{analytic estimates} of parametric solutions.
Roughly, there are two avenues for obtaining such bounds:
first, a real-variable bootstrapping argument as used in \cite{HoangSchwab14,BCDM2016,HKS21},
and
second, analytic continuation into the complex domain with precise control on the size
of the domain of analyticity \cite{DNSZ2023}, within the framework of holomorphy on
Banach spaces \cite{Munoz1999,MujiCplxAnBSpc}.
This regularity is key in establishing convergence rate bounds of $n$-term and sparse-grid
approximations by Wiener-Hermite polynomial chaos expansions, but also
on certain Quasi-Monte Carlo high-dimensional integration schemes,
all of whose rates are
not subject to the curse-of-dimensionality \cite{GKNS15_1104}.
\subsection{Contributions}
\label{sec:Contr}
  In this work, we develop the following contributions.
  \begin{enumerate}
  \item \label{item:contrib-summability}In an abstract setting,
    for holomorphic maps between separable Hilbert spaces,
    we obtain weighted summability (Theorem \ref{thm:bdXSum})
    and $N$-term approximation (Theorem \ref{thm:ubsnubound}) rate results
    which are free from the Curse of Dimensionality (``CoD'').
    These results hold in an extended, compared to those in \cite{DNSZ2023},
    range of summability exponents,
    albeit under a hypothesis of countably-parametric holomorphy
    which is slightly stronger than the one in \cite{DNSZ2023}.
    As in \cite{DNSZ2023}, the present results are not limited to scalar elliptic
        divergence-form PDEs, but rather apply to a wide range
        of operator equations with log-gaussian input data.
  \item \label{item:contrib-operators}
    We verify the abstract assumptions of the preceding result for
    countably parametric functions
    related to a large class of holomorphic operators between Hilbert spaces
    (Proposition \ref{prop:holomorphy}).
  \item\label{item:contrib-pde}
    We derive results on $p$-summability for $0<p<2$ of coefficient sequences in
      Wiener-Hermite polynomial chaos expansions of solutions
      of linear, second-order, divergence-form elliptic PDEs with log-Gaussian coefficients
      (Theorem \ref{thm:diffusion-holomorphy}).
\end{enumerate}
When applied to the operators of item \ref{item:contrib-operators} and the
  PDEs of item \ref{item:contrib-pde},
  the results that we prove in this work improve upon those in
  \cite{DNSZ2023}. We obtain, indeed, faster convergence rates for best $N$-term
  approximation and an increased
  summability range, as outlined in Remarks \ref{remark:convergence-rates} and \ref{rmk:CmpBCDM2016}.
\subsection{Layout}
\label{sec:outline}
The structure of this article is as follows.
In Section~\ref{sec:GMWgtSpc}, we set notation, recap known results on Gaussian
measures on function spaces from, e.g., \cite{DPZ2014,Bogachev1998}, and
introduce an extended (with respect to \cite{DNSZ2023})
notion of countably parametric holomorphy.
Section~\ref{sec:InfParHol} contains the definitions of parametric holomorphy
upon which our $n$-term approximation rate bounds are based.
In Section~\ref{sec:Spars},
we prove a summability result for Wiener-Hermite chaos expansion coefficients
of countably-parametric holomorphic functions.
This
result is used in Section~\ref{sec:hermite-appx} to prove approximation rates
for best $N$-term Hermite approximations.
In Section \ref{sec:ParHolOp},
we verify the assumptions in the abstract results
from Sections~\ref{sec:InfParHol}-\ref{sec:hermite-appx}
for a class of holomorphic operators between
separable Hilbert spaces (satisfying Assumption \ref{ass:G}).
The appendices contain
auxiliary results and proofs for
the main contributions of this work.
\subsection{Notation}
\label{sec:Notat}
We indicate the set of positive integers by $\N$
and write $\N_0 = \N \cup \{ 0 \}$.
We denote,
for $N\in \N$ and
for $\bsvarrho = ( \varrho_i)_{i=1,...,N} \in (0,\infty)^N$,
\begin{equation*}
  \cS_\bsvarrho^\C:=\set{\bsz\in\C^N}{|\Im(z_i)|<\varrho_i~\forall i}\subseteq\C^N.
\end{equation*}
We denote\footnote{Here, and throughout this article,
with slight abuse of notation,
we write countable cartesian products of a set $A$ by $A^\infty$, with the symbol
$\infty$ denoting, in this context, the $\infcard$ of G. Cantor, as e.g. in \cite{Bogachev1998}.
}
by $\calF$ sequences in
$\N_0^{\infty} = \bigtimes_{i\in\N}\N_0$
\emph{with finite support},
i.e.,
for each $\bsnu = (\nu_j)_{j\geq 1} \in \calF$,
only finitely many $\nu_j$ are nonzero.

For $M\in \bbN$ and for sequences
$\bsvarrho = (\varrho_j)_{j\geq 1} \subset (0,\infty)^\infty$,
we define
\begin{equation}
  \label{eq:beta}
  \beta_\bsnu(M,\bsvarrho)
  \coloneqq
  \sum_{\|\bsnu'\|_{\ell^\infty}\leq M} \binom{\bsnu}{\bsnu'} \bsvarrho^{2\bsnu'}
  = \prod_{j \in \N}\Bigg(\sum_{\ell=0}^{M}\binom{\nu_j}{\ell}\varrho_j^{2\ell}\Bigg), \quad \bsnu \in \calF
  \,.
\end{equation}

\section{Preliminaries}
\subsection{Gaussian measures and spaces}
\label{sec:GMWgtSpc}
We denote the product Gaussian measures (see \cite[page 372]{Bogachev1998}) on $\R^N = \bigtimes_{i=1}^N\R $
and
$\R^\infty  = \bigtimes_{i\in\N}\R$, respectively,
by
\begin{equation*}
\gamma_N = \bigotimes_{i=1}^N\calN(0, 1), \quad\text{and}\quad  \gamma = \bigotimes_{i\in \N}\calN(0, 1).
\end{equation*}
Then,
$L^2_\gamma(\R^\infty)$ is the space of real-valued, square $\gamma$-summable
functions, i.e., functions $f$ with finite norm given by
\begin{equation*}
  \| f\|_{L^2_\gamma(\R^\infty)} ^2 \coloneqq \int_{\R^\infty} |f(\bsy)|^2\gamma(d\bsy) < \infty,
\end{equation*}
with the obvious modification for functions defined on $\R^N$.

For a separable Hilbert space $X$, the Bochner space of $X$-valued
functions $L^2_\gamma(\R^\infty,X) $  satisfies
\begin{equation*}
L^2_\gamma(\R^\infty,X) = L^2_\gamma(\R^\infty)\otimes X
\end{equation*}
with Hilbertian tensor product.
We equip $L^2_\gamma(\R^\infty,X)$ with the norm given by
\begin{equation*}
  \| v \|_{L^2_\gamma(\R^\infty, X)}^2 \coloneqq \int_{\R^\infty} \| v(\bsy)\|^2_X \gamma(d\bsy).
\end{equation*}
\subsection{Infinite parametric holomorphy}
\label{sec:InfParHol}
We introduce the main subject of this paper, i.e.,
countably parametric holomorphic functions.
Compared with the definition given, e.g., in \cite[Definition 4.1]{DNSZ2023},
we introduce an additional parameter %
$r\in [1,\infty]$
in the definition of admissibility of a
weight sequence $\bsvarrho \in (0, \infty)^\infty$.
As a consequence,
the same parameter $r$ is present in the
definition of ($\bsb,\xi,r, \delta, X$)-holomorphy.

The definition in \cite[Definition 4.1]{DNSZ2023} corresponds
to the case $r=1$, while in this paper we are concerned with $r=2$.
By choosing $r=2$ we obtain a stronger regularity requirement,
which will be instrumental in
proving the sparsity bounds in Theorem \ref{thm:bdXSum} and
the approximation rates of Theorem \ref{thm:ubsnubound}.

The justification for imposing the
stronger requirement $r=2$ is given in Proposition \ref{prop:holomorphy} and in
Theorem \ref{thm:diffusion-holomorphy}, where we show that some functions
related to the solution operator of elliptic PDEs are
($\bsb,\xi, 2, \delta, X$)-holomorphic.
A comparison between the notions of
holomorphy for different values of $r$ is given in Remark
\ref{remark:r-comparison} below.
\begin{definition}
\label{def:Adm}
  Let  $\bsb=(b_j)_{j\in\N}\in (0,\infty)^\infty$,
  $\xi>0$, $\delta>0$, and $r\in [1,\infty]$.
  We say that $\bsvarrho\in (0,\infty)^\infty$
  is $(\bsb,\xi,r)$-admissible if
  \begin{equation*}
    \sum_{j\in\N} b_j^r \varrho_j^r \leq \xi^r \;\;\text{if}\;\; r<\infty,
    \qquad \max_{j\in\N} b_j\varrho_j \leq \xi \;\;\text{if}\;\; r = \infty.
  \end{equation*}
\end{definition}

\begin{definition}[($\bsb,\xi,r, \delta, X$)-Holomorphy]
\label{def:bdXHol}
  Let $X$ be a complex, separable Hilbert space.
  Let
  $\bsb=(b_j)_{j\in\N} \in (0,\infty)^\infty$, $\xi>0$, $\delta>0$, and $r\in [1,\infty]$.
  A function $u\in L^2_\gamma(\R^\infty,X)$ is called
  \emph{$(\bsb,\xi,r, \delta, X)$-holomorphic}
  if
  \begin{enumerate}
  \item\label{item:hol} for every $N\in\N$ there exists
    $u_N:\R^N\to X$, which, for every $(\bsb,\xi, r)$-admissible
    $\bsvarrho\in (0,\infty)^N$, admits a holomorphic extension
    (denoted again by $u_N$) from $\cS_\bsvarrho^\C\to X$; furthermore,
    for all $N<M$
    \begin{equation}\label{eq:un=um}
      u_N(y_1,\dots,y_N)=u_M(y_1,\dots,y_N,0,\dots,0)\qquad\forall (y_j)_{j=1}^N\in\R^N,
    \end{equation}

  \item\label{item:varphi} for every $N\in\N$ there exists
    $\varphi_N:\R^N\to\R_+$ such that
    $\norm[L^2_{\gamma_N}(\R^N)]{\varphi_N}\le\delta$ and
    \begin{equation*} \label{ineq[phi]}
      \sup_{\substack{\bsvarrho\in(0,\infty)^N\\
          \text{is $(\bsb,\xi)$-adm.}}}~\sup_{\bsz\in
        \calB(\bsvarrho)}\norm[X]{u_N(\bsy+\bsz)}\le
      \varphi_N(\bsy)\qquad\forall\bsy\in\R^N,
    \end{equation*}
  \item\label{item:vN} with $\tilde u_N:\R^\infty \to X$ defined by
    $\tilde u_N(\bsy) :=u_N(y_1,\dots,y_N)$ for $\bsy\in \R^\infty$
  it holds that
    \begin{equation*}
      \lim_{N\to\infty}\norm[L^2_\gamma(\R^\infty,X)]{u-\tilde u_N}=0.
    \end{equation*}
  \end{enumerate}
\end{definition}
\begin{remark}
\label{remark:r-comparison}
A sequence $\bsvarrho\in (0, \infty)^\infty$ is denoted
$(\bsb, \xi)$-admissible in \cite[Definition 4.1]{DNSZ2023}
when it is $(\bsb, \xi,1)$-admissible in the notation of the present paper.
When $\bsb\in \ell^1(\bbN)$, if
$\bsvarrho$ is $(\bsb, \xi, r)$-admissible,
then it is also $(\bsb, \widetilde\xi, \widetilde{r})$-admissible
for $1\leq \widetilde{r}\leq r\leq \infty$ and for some value of $\widetilde\xi$.

Similarly, a ($\bsb,\xi, \delta, X$)-holomorphic function following
\cite[Definition 4.1]{DNSZ2023} is denoted here ($\bsb,\xi, 1, \delta, X$)-holomorphic.
It immediately follows from the %
argument above that %
$(\bsb, \xi, r, \delta,X)$-holomorphy implies
$(\bsb, \widetilde\xi, \widetilde{r}, \delta, X)$-holomorphy for all
$1\leq \widetilde{r}\leq r\leq \infty$ and for some value of $\widetilde\xi$.
\end{remark}

\section{Sparsity for $(b, \xi, 2, \delta, X)$-holomorphic functions}
\label{sec:Spars}
We recall the isometric isomorphy
$
L^2_\gamma(\R^\infty,X) = L^2_\gamma(\R^\infty)\otimes X
$
with Hilbertian tensor product.
Any $u\in L^2_\gamma(\R^\infty, X)$
can be represented in a Wiener-Hermite polynomial chaos (PC) expansion
\begin{equation}\label{eq:uHermiteExp}
  u=\sum_{\bsnu\in \calF} u_\bsnu H_\bsnu,
\end{equation}
where
\begin{equation}
  \label{eq:unu}
  u_\bsnu=\int_{\R^\infty} u(\bsy) H_\bsnu(\bsy)  \gamma(d\bsy)
\end{equation}
are the Wiener-Hermite PC expansion coefficients.
Also, we have the Parseval-type identity
\begin{equation*}\label{eq:PCParseval}
  \|u\|_{L^2_\gamma(\R^\infty,X)}^2=\sum_{\bsnu \in \cF}
  \|u_\bsnu\|_X^2\,.
\end{equation*}
In  the following statement, we study the summability of the Wiener-Hermite
coefficients of a $(\bsb,\xi,2, \delta,X)$-holomorphic function $u$. The proof
follows that of \cite[Theorem 4.9]{DNSZ2023}, which treats the case $r=1$.
The stronger hypothesis on $u$ imposed here implies a stronger result on the
summability of the coefficients, which is instrumental for the approximation
result given in Theorem \ref{thm:ubsnubound}.
\begin{theorem}\label{thm:bdXSum}
Let $u$ be
$(\bsb,\xi,2, \delta,X)$-holomorphic for some
$\bsnul \neq \bsb \in \ell^p(\N)$, with a $p\in (0,2)$.
Let $M\in\N$ with $M\geq 2/p$, and
let $\bsvarrho \in (0, \infty)^\infty$
be such that
 \begin{equation} \label{eq:varrho_j}
   \varrho^2_jb_j^2:=\left( \frac{b_j}{\norm[\ell^p]{\bsb}} \right)^p\frac{\xi^2}{8eM! }, \quad \forall
  j\in\N.
\end{equation}
 Then,
  	\begin{equation} \label{eq:general}
  		\sum_{\bsnu\in\calF}\beta_\bsnu(M,\bsvarrho)\|u_\bsnu\|_X^2
  		\le
  		2\delta^2, \qquad
  		 \big(\beta_\bsnu(M,\bsvarrho)^{-1}\big)_{\bsnu\in\calF} \in \ell^{p/(2-p)}(\calF),
  	\end{equation}
  	 where the weight sequence $ \beta_\bsnu(M,\bsvarrho)$ is defined in \eqref{eq:beta}.
  	Furthermore,
  	$$
  	(\|u_\bsnu\|_X)_{\bsnu\in\calF}\in \ell^{p}(\calF).
  	$$
\end{theorem}
\begin{proof}
  Let $u_N\in L^2_\gamma(\R^\infty,X)$ be any function satisfying items
  \ref{item:hol}--\ref{item:vN} of
  Definition~\ref{def:bdXHol}. Similarly to \eqref{eq:uHermiteExp}, we write
  $$
  u_{N,\bsnu}:= \int_{\R^\infty} u_N(\bsy) H_\bsnu(\bsy)\gamma(d\bsy)
  \in X
  $$
  for any $\bsnu\in\calF$.
By definition, $u_N\to u\in L^2_\gamma(\R^\infty,X)$, hence
  $$\lim_{N\to\infty} u_{N,\bsnu}=u_{\bsnu}$$
  for all
  $\bsnu\in\calF$.
  Furthermore, since $u_N$ is independent of the variables
  $(y_j)_{j=N+1}^\infty$ and by the orthogonality to constants of Hermite polynomials $H_k$,
  $k>0$, we have $u_{N,\bsnu}=0$ whenever
  $\supp\bsnu\not\subseteq \{1,\dots,N\}$. It follows that
  \begin{equation}
    \label{eq:bound-uNnu}
    \sum_{\bsnu\in\calF}\beta_\bsnu(M,\bsvarrho) \norm[X]{u_{N,\bsnu}}^2
=\sum_{\bsnu\in\bbN_0^N}\beta_\bsnu(M,\bsvarrho_N) \norm[X]{u_{N,\bsnu}}^2
    \le
    2\delta^2\qquad\forall N\in\N,
  \end{equation}
 where the second inequality follows from Proposition~\ref{prop:bdHolSum} and
 $\bsvarrho_N$ is the $N$-term truncation of $\bsvarrho$.
  For any finite $\Lambda\subset\calF$, then,
  \begin{equation}
    \label{eq:bound-unu-Lambda}
    \sum_{\bsnu\in\Lambda}
    \beta_\bsnu(M,\bsvarrho) \norm[X]{u_{\bsnu}}^2
    =\lim_{N\to\infty}\sum_{\bsnu\in\Lambda}
    \beta_\bsnu(M,\bsvarrho) \norm[X]{u_{N,\bsnu}}^2\le 2\delta^2.
  \end{equation}
  Since $\Lambda\subset\calF$ was arbitrary and finite,
  combining \eqref{eq:bound-uNnu} and \eqref{eq:bound-unu-Lambda} gives
  \begin{equation} \nonumber
  	\sum_{\bsnu\in\calF}\beta_\bsnu(M,\bsvarrho)\|u_\bsnu\|_X^2
  	\le
  	2\delta^2.
  \end{equation}
  Finally,
due to $\bsb\in (0,\infty)^{\infty} \cap \ell^p(\NN)$ with some $0<p<2$,
definition \eqref{eq:varrho_j} implies directly that
  $(\varrho_j^{-1})_{j\in\N} \in \ell^{2p/(2-p)}(\N)$.
  By  \cite[Lemma 3.11]{DNSZ2023}
  \begin{equation} \label{beta_nu-summability}
  (\beta_\bsnu(M,\bsvarrho)^{-1})_{\bsnu\in\calF}\in\ell^{p/(2-p)}(\calF),
  \end{equation}
  which concludes the proof of \eqref{eq:general}.

To prove the $p$-summability of $\norm[X]{u_{N,\bsnu}}$ we use H\"older's inequality to infer
  \begin{align*}
    \sum_{\bsnu\in\calF}\norm[X]{u_{N,\bsnu}}^{p}
    &= \sum_{\bsnu\in\calF}\norm[X]{u_{N,\bsnu}}^{p}\beta_{\bsnu}(M,\bsvarrho_N)^{\frac p 2}\beta_{\bsnu}(M,\bsvarrho_N)^{-\frac p 2}\nonumber\\
    &\le \Bigg(\sum_{\bsnu\in\calF}\norm[X]{u_{N,\bsnu}}^{2}\beta_{\bsnu}(M,\bsvarrho_N) \Bigg)^{\frac p 2}\Bigg(\sum_{\bsnu\in\N_0^N}\beta_{\bsnu}(M,\bsvarrho_N)^{\frac{p}{p-2}} \Bigg)^{\frac{2-p}{2}}.
  \end{align*}
  Hence using \eqref{beta_nu-summability} we obtain $(\norm[X]{u_\bsnu})_{\bsnu\in\calF}\in\ell^{p}(\calF)$.
\end{proof}

\section{Constructive Hermite approximation}
\label{sec:hermite-appx}
We now discuss how to use our analysis for constructing approximations
based on truncated Hermite expansions, or sparse-grid interpolation
based on Gauss-Hermite interpolation points. We start by recalling\footnote{In \cite{DNSZ2023}, the result assumes a certain parameter $\tau$ to be positive; the arguments remain valid for $\tau = 0$ however, which is the case considered in Lemma \ref{lemma:cnu}. For completeness, we include a proof of Lemma \ref{lemma:cnu} in Appendix~\ref{sec:proofcnu}, noting that it follows \cite{DNSZ2023} with only minor modifications.}
\cite[Lemma 6.5]{DNSZ2023} and \cite[Lemma 6.6]{DNSZ2023}.

\begin{lemma}\label{lemma:cnu}
  Let $M>1$, $\bsvarrho\in (0,\infty)^\infty$ and
  let $\beta_{\bsnu}(M,\bsvarrho)$ be as in \eqref{eq:beta}.
  Denote, for any $\bsnu\in\calF$,
  \begin{equation}\label{eq:cbsnu}
    c_{\bsnu}=\prod_{j\in\supp\bsnu}\max\{1,K\varrho_j\}^2\nu_j^{M}.
  \end{equation}
  Then, there exist $K$, $C_\beta>0$ depending on
  $M$ and $\bsvarrho$ such that,
  \begin{equation}
    \label{eq:cnu-betanu}
    C_\beta c_\bsnu\le \beta_\bsnu(M,\bsvarrho)\qquad\forall\bsnu\in\calF.
  \end{equation}
  Moreover, if $q>\frac{2}{M}$ and $(\varrho_j^{-1})_{j\in\N}\in\ell^q(\N)$, then
  \begin{equation*}
    (c_\bsnu^{-1})_{\bsnu\in\calF}\in\ell^{q/2}(\calF).
  \end{equation*}
\end{lemma}

The sequence $(c_\bsnu)_{\bsnu\in\calF}$ can be used to approximate
$u$ via thresholding on these coefficients: let $(\bsnu_i)_{i\in\N}$
any enumeration of $\calF$ such that
$(c_{\bsnu_i})_{i\in \bbN}$ is monotonically increasing in the sense that
\begin{equation}
  i\leq j\quad\implies\quad
  c_{\bsnu_i}\le c_{\bsnu_j}\qquad\forall i,j\in\bbN.
\end{equation}
Then, we can approximate
\begin{equation}\label{eq:LambdaN}
  u(\bsy)\simeq \sum_{\bsnu\in\Lambda_N}u_\bsnu H_\bsnu(\bsy)
\qquad\text{where}\qquad
  \Lambda_N:=\set{\bsnu_i}{i\le N}.
\end{equation}
The rationale behind \eqref{eq:LambdaN} is that $c_\bsnu$ gives an
indication on the size of the Hermite coefficient $u_\bsnu$ due to
the summability property
\begin{equation*}
  \sum_{\bsnu\in\calF}c_\bsnu\norm[X]{u_\bsnu}^2<\infty
\end{equation*}
which is an immediate consequence of Theorem \ref{thm:bdXSum} and
Lemma \ref{lemma:cnu}.
We obtain the following \textit{a priori} convergence-rate
bound on this approximation.
\begin{theorem}\label{thm:ubsnubound}
  Let $u$ be $(\bsb,\xi,2, \delta,X)$-holomorphic with
  $\bsb\in \ell^p(\N)$ for some $p\in (0,2)$.
  Let
  $M\in\N$ with
  $M>2/p$, let $\varrho_j$ be as in \eqref{eq:varrho_j},
  and let
  $\Lambda_N$ be as in \eqref{eq:LambdaN}.

  Then
  \begin{equation*}
  \left\|u(\bsy)-\sum_{\bsnu\in\Lambda_N}u_\bsnu H_\bsnu(\bsy)\right\|_{L^2_\gamma(\R^\infty,X)}
  =
  \calO(N^{-\frac{1}{p}+\frac{1}{2}})\qquad\text{as }N\to\infty.
  \end{equation*}
\end{theorem}
\begin{proof}
  Since $\bsb\in\ell^p$, the definition of $\varrho_j$ in \eqref{eq:varrho_j}
  implies $\varrho_j^{-1}\sim b_j^{(2-p)/2}$ so that
  $(\varrho_j^{-1})_{j\in\N}\in\ell^{\tilde q}(\N)$ with $\tilde q=\frac{2p}{2-p}$.
  According to Lemma \ref{lemma:cnu} we thus have with $q=\frac{p}{2-p}$
  \begin{equation}\label{eq:cbsnuellq}
    (c_\bsnu^{-1})_{\bsnu\in\calF}\in\ell^{q}(\calF).
  \end{equation}

  Using Parseval's identity and the fact that
  $u(\bsy) = \sum_{\bsnu\in\calF}u_\bsnu H_\bsnu(\bsy)$ converges in $L^2_\gamma(\R^\infty,X)$,
  we find
  \begin{align*}
    \left\|u(\bsy)-\sum_{\bsnu\in\Lambda_N}u_\bsnu H_\bsnu(\bsy)\right\|_{L^2_\gamma(\R^\infty,X)}^2
    &= \sum_{\bsnu\in\calF\backslash\Lambda_N}\norm[X]{u_\bsnu}^2 (c_\bsnu c_\bsnu^{-1})\\
    &\le \left(\max_{\bsmu\in\calF\backslash\Lambda_N} c_\bsmu^{-1}  \right)\sum_{\bsnu\in\calF}\norm[X]{u_\bsnu}^2 c_\bsnu.
  \end{align*}
  The last sum is finite by the inequalities in \eqref{eq:general} and
  \eqref{eq:cnu-betanu}.
  Due to \eqref{eq:cbsnuellq} and the fact that
  $\max_{\bsxi\in\calF\backslash\Lambda_N} c_\bsxi^{-1}$ is the $N$th
  largest element of $\set{c_\bsnu^{-1}}{\bsnu\in\calF}$ by definition
  of $\Lambda_N$, Stechkin's lemma yields the bound
  \begin{equation*}
    \max_{\bsxi\in\calF\backslash\Lambda_N} c_\bsxi^{-1}\le N^{-1/q}\norm[\ell^q]{(c_\bsxi^{-1})_{\bsxi\in\calF}}.
  \end{equation*}
  Since $\frac{1}{q}=\frac{2}{p}-1$,
  this proves the claim.
\end{proof}

Working with the weight sequence $(c_\bsnu)_{\bsnu\in\calF}$ instead
of directly with the $(\beta_\bsnu)_{\bsnu\in\calF}$ has the advantage
that its simple structure allows for the efficient computation of
$\Lambda_N$, which on top is guaranteed to be a downward closed
set. Indeed, different algorithms of (essentially) linear complexity
determining such sets $\Lambda_N$ can be found in
\cite{MR2566594,JZdiss,Westermann2023Transport}.

Finally, for future reference, we give a structural result on the
multiindex sets in \eqref{eq:LambdaN} regarding the growth of the effective
intrinsic dimensionality; specifically, we are interested in the
behaviour of
\begin{equation*}
  m(\Lambda_N):=\max_{\bsnu\in\Lambda_N}\|\bsnu\|_1\qquad\text{and}\qquad
  d(\Lambda_N):=\max_{\bsnu\in\Lambda_N}\|\bsnu\|_0.
\end{equation*}
The following lemma bounds the growth of these terms as $N\to\infty$.
It thus characterizes a form of sparsity of these index sets.
We recap the argument here for convenience of the reader,
and
mention that a similar bound (formulated slightly differently)
was shown in \cite[Lemma 5.5]{SZ21_982}.
\begin{lemma}
Let $C>0$, $0<r_0<r_1<\infty$ and
let $c_\bsnu$ be defined as in \eqref{eq:cbsnu}.
Let furthermore
  \begin{enumerate}
  \item $(\varrho_j)_{j\in\N}\subset (0,\infty)$
        in the definition of $c_\bsnu$ be monotonically
        increasing and $C^{-1} j^{r_0}\le \varrho_j\le C j^{r_1}$ for all $j\in\N$,
    \item $(\bsnu_j)_{j\in\N}$ be any enumeration of $\calF$ such that
      the sequence $(c_{\bsnu_j})_{j\in\N}$ in \eqref{eq:cbsnu} is monotonically increasing,
    \item $\Lambda_N$ be as in \eqref{eq:LambdaN}.
\end{enumerate}
Then, as $N\to\infty$,
  \[
    m(\Lambda_N) =O(N^{2r_1/M}), \qquad
    d(\Lambda_N) = o(\log(N)) \;.
  \]
\end{lemma}

\begin{proof}
  We start with $d(\Lambda_N)$. By definition of $c_\bsnu$ in \eqref{eq:cbsnu},
  and due to the $\varrho_j$ being monotonically increasing, the minimum
  value of $c_\bsnu$ for any $\bsnu$ with $\|\bsnu\|_0\ge d$ is clearly attained
  at $\bsnu=(1,\dots,1,0,0,\dots)$, where the first $d$ entries equal one. Hence
  \begin{equation*}
    \min_{\|\bsnu\|_0\ge d}c_\bsnu \ge \prod_{j=1}^d\max\{1,K\rho_j\}^2
    \ge \prod_{j=1}^d(CK j^{r_0})^2
    \ge (CK)^{2d} (d!)^{2r_0}
  \end{equation*}
  which increases faster than exponential in $d$. On the other hand,
  using that $\varrho_j$ is monotonically increasing and
  $\varrho_j\le C j^{r_1}$, by \eqref{eq:cbsnu}
  \begin{equation}\label{eq:cbsnuN}
    c_{\bsnu_N}\le \max\{1,K\varrho_N\}^2
    \le \max\{1,CK N^{r_1}\}^2
    \le \max\{1,CK\}^2  N^{2 r_1}
  \end{equation}
  which grows algebraically in $N$. Together this gives that
  \begin{equation*}
    (CK)^{2d(\Lambda_N)}(d(\Lambda_N)!)^{2r_0}\le \max\{1,CK\}^2 N^{2 r_1}
  \end{equation*}
  and therefore necessarily
  \begin{equation*}
    d(\Lambda_N)=o(\log(N))\qquad\text{as }N\to\infty.
  \end{equation*}

  For $m(\Lambda_N)$ we proceed similarly. First, given a multi index $\bsnu$,
  we observe that for the $j$th unit index $\bse_j=(\delta_{ij})_{i\in\N}$,
  $c_{\bsnu+\bse_j}$ increases with respect to $c_\bsnu$ by the factor
  \begin{equation*}
    \frac{c_{\bsnu+\bse_j}}{c_\bsnu} =\begin{cases}
      \max\{1,K\rho_j\}^{2}&\text{if }\nu_j=0\\
      \frac{(\nu_j+1)^M}{\nu_j^M}&\text{if }\nu_j>0.
    \end{cases}
  \end{equation*}
  Let $d_0$ be so large that $\max\{1,K\varrho_{d_0}\} = K\varrho_{d_0}\ge
  \max_{n\in\N}\frac{(n+1)^M}{n^M}=2^M$.
  Since $\frac{(n+1)^M}{n^M}$ is decreasing in $n$, one multiindex $\bsnu$ minimizing $c_\bsnu$
  under the constraint $\|\bsnu\|_1=m$ must be of the type
  \begin{equation*}
    \bsnu=(m-k_0,\underbrace{1,\dots,1}_{k_0\text{ times}},0,0,\dots)
  \end{equation*}
  for some $k_0\le d_0-1$. We conclude that
  \begin{equation*}
    \min_{\|\bsnu\|_1\ge m}c_\bsnu \ge \max\{1,K\rho_1\}^2(m-k_0)^M
    \prod_{j=2}^{k_0+1}\max\{1,K\rho_j\}^2
    \ge \max\{1,m-d_0\}^M.
  \end{equation*}
  Using this together with \eqref{eq:cbsnuN} we have
  \begin{equation*}
    \max\{1,m(\Lambda_N)-d_0\}^M \le \max\{1,CK\}^2 N^{2r_1}
  \end{equation*}
  which gives
  \begin{equation*}
    m(\Lambda_N)
    \le \max\{1,CK\}^{2/M} N^{2r_1/M}+d_0 =O(N^{2r_1/M})
  \end{equation*}
  for $N\to\infty$ as claimed.
\end{proof}

\section{Parametric holomorphy for operators between Hilbert spaces}
\label{sec:ParHolOp}
Let $\calX$ and $\calY$ be real-valued, separable Hilbert
spaces.
Let $(\psi_j)_{j\in\N}$ be a Riesz
basis for $\calX$: there exists
a constant $\Cpsi\geq 1$ such that for every $\bsx \in \ell^2(\N)$
\begin{equation}\label{eq:RszBas}
\frac{1}{\Cpsi^2} \| \bsx \|^2_{\ell^2(\N)} \leq \| \sum_{j\in\N}x_j\psi_j \|^2_{\calX} \leq \Cpsi^2 \| \bsx \|^2_{\ell^2(\N)},
\end{equation}
We write
\begin{equation*}
  \calS_\theta^\calX := \set{x+\ii y}{x,y\in\calX,~\norm[\calX]{y}<\theta} \subset \calX_\bbC
\end{equation*}
to denote a complex ``strip'' of width $\theta>0$
in the (canonical, see, e.g., \cite{Michal1940,MujiCplxAnBSpc,padraig,Munoz1999} or \cite[Sec.~1.2.4]{JZdiss})
complexification $\calX_\bbC$ of $\calX$.

In this section we will consider maps $\calG:\calX_{\bbC}\to\calY_{\bbC}$
that satisfy the following Assumption \ref{ass:G}.
We will show that this assumption implies
holomorphy of countably parametric functions in the sense of Definition~\ref{def:bdXHol}
which are related to the map $\calG$ as in \eqref{eq:uy-def} below.
\begin{assumption}\label{ass:G}
  There exist $\theta > 0$
  such that $\calG:\calS_\theta^\calX\to \calY_{\bbC}$
  is holomorphic and, for some $C, \alpha>0$, and for some $\tau\in [0,2)$,
  it holds that
  \begin{enumerate}
  \item \label{item:G-bound}$\norm[\calY_{\bbC}]{\calG(a)}\le C\exp(\alpha\norm[\calX_{\bbC}]{a}^\tau)$,
  \item \label{item:G-cont}$\norm[\calY_{\bbC}]{\calG(a)-\calG(b)}\le C \exp(\alpha(\norm[\calX_{\bbC}]{a}^\tau+\norm[\calX_{\bbC}]{b}^\tau))\norm[\calX_{\bbC}]{a-b}$
  \end{enumerate}
  for all $a$, $b\in\calS_\theta^\calX$.
\end{assumption}
\begin{remark}
  Item \ref{item:G-cont} of Assumption \ref{ass:G} implies that there exist
  $\widetilde{C}$ and $\widetilde{\alpha}$ such that for all $a\in\calS_\theta^\calX$
  \begin{equation*}
    \norm[\calY_{\bbC}]{\calG(a)}\le \widetilde{C}\exp(\widetilde{\alpha}\norm[\calX_{\bbC}]{a}^\tau),
  \end{equation*}
  i.e., item
  \ref{item:G-bound} of the same assumption, with modified constants. We list
  the two conditions anyway for ease of notation in the rest of the manuscript.
\end{remark}
\begin{remark}
  \label{remark:holomorphy-embedding}
  To satisfy Assumption \ref{ass:G}, it is sufficient that the
  operator $\calG$ is holomorphic $\calS^\calX_\theta \to Z$, where
  $Z\hookrightarrow \calY_\bbC$ is a complex Hilbert space continuously embedded
  in $\calY_\bbC$. Indeed,
 $\calG$ is continuous from $\calS^\calX_\theta$
 to $\calY_\bbC$ by item \ref{item:G-cont}.
 It follows then from \cite[Proposition 1]{HSZ22_1015} that
$\calG : \calS^\calX_\theta\to \calY_\bbC$
is holomorphic.
\end{remark}
Assumption \ref{ass:G} is a sufficient condition for parametric $(\bsb, \xi, 2,
\delta, \calY_{\bbC})$ holomorphy in the sense of Definition~\ref{def:bdXHol},
as formulated in the next statement.
\begin{proposition}\label{prop:holomorphy}
Suppose that $\calG$ satisfies Assumption \ref{ass:G} and let $\bsb\in\ell^2(\N)$ be a sequence of positive reals. Then,
\begin{equation}
  \label{eq:uy-def}
  u(\bsy):=\calG\Big(\sum_{j\in\N}y_jb_j\psi_j\Big)
\end{equation}
is $(\bsb,\theta/\Cpsi, 2 , \delta, \calY_\bbC)$-holomorphic
for some $\delta=\delta(\alpha,\bsb,C_\psi,C_\tau,\tau)$.
\end{proposition}
\begin{proof}
 By Assumption \ref{ass:G},
 $\calG$ is holomorphic from $\calS_\theta^{\calX}$ to $\calY_{\bbC}$.

We denote for arbitrary, fixed $N\in \N$ and for $y_j\in \R$
\begin{equation*}
  u_N(y_1, \dots, y_N)
  =
  \calG\Big(\sum_{j=1}^N y_jb_j\psi_j\Big), \qquad \tilde u_N(\bsy) = u_N(y_1, \dots, y_N).
\end{equation*}
\textbf{Verification of item \ref{item:hol} of Definition \ref{def:bdXHol}.}
Let
$\bsvarrho\in (0, \infty)^N$ be $(\bsb, \theta/\Cpsi, 2)$ admissible
and consider the extension
\begin{equation*}
  u_N(\bsz) = \calG\Big(\sum_{j=1}^N z_j b_j \psi_j\Big).
\end{equation*}
We have, for all $\bsz \in \calS^{\C}_{\bsvarrho}$,
\begin{equation*}
  \| \sum_{j=1}^N\Im(z_j)b_j\psi_j \|^2_{\calX}
  \leq \Cpsi^2 \sum_{j=1}^N |\Im(z_j)|^2b_j^2
  \leq \Cpsi^2 \sum_{j=1}^N \varrho_j^2b_j^2
  \leq \theta^2.
\end{equation*}
Hence, for all $\bsz \in \calS^{\C}_{\bsvarrho}$,
\begin{equation*}
  \sum_{j=1}^N z_jb_j\psi_j\in \calS^\calX_\theta
\end{equation*}
and $u_N$ is holomorphic from $\calS^{\C}_{\bsvarrho}$ to $\calY_{\bbC}$.
\\
\textbf{Verification of item \ref{item:varphi} of Definition \ref{def:bdXHol}.}
For $\tau \in [0,2)$ as in Assumption \ref{ass:G},
let
$C_\tau$ be a constant such that
$(x+y)^{\tau/2}\leq C_\tau(x^{\tau/2} + y^{\tau/2})$ for all $x, y\in \R_+$.
For every $\bsvarrho\in (0, \infty)^N$ which is $(\bsb, \theta/\Cpsi, 2)$
admissible, for all $\bsy\in \R^N$, and for all $\bsz\in \calB(\bsvarrho)$, we estimate
\begin{align*}
  \| u(\bsy+\bsz) \|_{\calY_{\bbC}}
  &\leq C\exp \left(\alpha \| \sum_{j=1}^N(y_j+z_j)b_j\psi_j \|^\tau_{\calX_{\bbC}} \right) \\
  & \leq C\exp\left( \alpha \Cpsi^\tau\left[ \sum_{j=1}^N (y_j + \Re(z_j))^2 b_j^2 + \sum_{j=1}^N \Im(z_j)^2 b_j^2 \right]^{\tau/2} \right)\\
  & \leq C\exp\left( \alpha \Cpsi^\tau\left[ 2\sum_{j=1}^N y_j^2 b_j^2 + \sum_{j=1}^N(2\Re(z_j)^2 + \Im(z_j)^2)b_j^2\right]^{\tau/2} \right)\\
  & \leq C\exp\left( \alpha \Cpsi^\tau\left[ 2\sum_{j=1}^N y_j^2 b_j^2 + 2\theta^2\Cpsi^{-2}\right]^{\tau/2} \right)\\
  & \leq Ce^{2^{\tau/2}\alpha C_\tau \theta^\tau}\exp\left( 2^{\tau/2}C_\tau \alpha \Cpsi^\tau\left[ \sum_{j=1}^N y_j^2 b_j^2 \right]^{\tau/2} \right) \eqqcolon \varphi_N(\bsy).
\end{align*}
By Lemma \ref{lemma:expL2},
$\varphi_N$ is in $L^2_{\gamma_N}(\R^N)$ with a bound $\delta=\delta(\alpha,\bsb,C_\psi,C_\tau,\tau)$ on the norm
that is independent of $N$.

\noindent\textbf{Verification of item \ref{item:vN} of Definition \ref{def:bdXHol}.}
Given the sequence $\bsb$, let $\tilde{\alpha}(\bsb) > 0$ be so that
$\tilde{\alpha} \; \sup_{j\in \N}b_j ^2 < 1/2$ %
and let
$C_{\alpha, \tau}>0$ be a constant such that
\begin{equation*}
e^{4\alpha \Cpsi^\tau x^{\tau/2}} \leq C_{\alpha, \tau} e^{\tilde{\alpha} x}, \qquad \forall x\in\R_+.
\end{equation*}
We start by showing that $u\in L_\gamma^2(\R^\infty, \calY)$:
\begin{align*}
   \int_{\R^\infty} \| u(\bsy) \|_{\calY_{\bbC}}^2 \gamma(d\bsy)
   &
     \leq \int_{\R^\infty} C^2 \exp\Big(2\alpha \| \sum_{j\in\N} y_j b_j \psi_j \|_{\calX_{\bbC}}^\tau\Big) \gamma(d\bsy)
   \\ &
        \leq C^2 \int_{\R^\infty} \exp\Big(2\alpha\Cpsi^\tau \Big( \sum_{j\in\N}b_j^2y_j^2 \Big)^{\tau/2}\Big) \gamma(d\bsy) < \infty,
 \end{align*}
where the final bound follows from Lemma \ref{lemma:expL2}.
We have
\begin{align*}
  \| u - \tilde u_N \|_{L_\gamma^2(\R^\infty, \calY_{\bbC})}^2
  &=  \int_{\R^\infty} \| u(\bsy) - \tilde u_N(\bsy) \|^2_{\calY_{\bbC}} \gamma(d\bsy)
  \\ &
       \leq C^2 \int_{\R^\infty} \exp\Big( 2\alpha \|\sum_{j\in\N} y_jb_j\psi_j\|_{\calX_{\bbC}}^\tau + 2\alpha\|\sum_{j=1}^N y_jb_j\psi_j\|_{\calX_\bbC}^\tau \Big) \| \sum_{j=N+1}^\infty y_jb_j\psi_j\|^2_{\calX_\bbC}\gamma(d\bsy)
  \\ &
\leq C^2 \Cpsi^2\int_{\R^\infty} \exp\Big(4\alpha\Cpsi^\tau \Big(\sum_{j\in\N}y_j^2b_j^2\Big)^{\tau/2}\Big)\sum_{j=N+1}^\infty b_j^2y_j^2 \gamma(d\bsy)
  \\ &
       \leq  C^2\Cpsi^2C_{\alpha, \tau} \int_{\R^\infty} \exp\Big(\tilde{\alpha} \sum_{j\in\N}y_j^2b_j^2\Big)\sum_{j=N+1}^\infty b_j^2y_j^2 \gamma(d\bsy)
  \\ &
       = C^2\Cpsi^2C_{\alpha, \tau} \left(\prod_{k\in\N}\left( \frac{1}{1-2\tilde{\alpha}b_k^2} \right)^{1/2}\right)\sum_{j=N+1}^\infty\frac{b_j^2}{1-2\tilde{\alpha}b_j^2}.
\end{align*}
The sum in the last term in the inequality converges to zero as $N\to \infty$, since
$\bsb\in \ell^2(\N)$.
It follows that $\tilde u_N \to u$ in $L_\gamma^2(\R^\infty, \calY_\bbC)$.
\end{proof}
\subsection{Application:
Coefficient-to-solution map for a linear elliptic PDE}
\label{sec:diffusion}
In this section, we will show that the coefficient-to-solution map of an
elliptic PDE  with lognormal coefficient satisfies Assumption \ref{ass:G} by
choosing $\calX$ and $\calY$ to be the appropriate Sobolev spaces.

We consider the extension to complex-valued spaces of the map
\begin{equation*}
\calG :
\begin{cases}
  L^\infty(\bbT^d) \to \{v\in H^1(\bbT^d): \int_{(0,1)^d} v = 0\}\\
  a\mapsto u
\end{cases},
\end{equation*}
where $u$ is the (weak)
solution to
\begin{equation}
  \label{eq:diffusion-torus}
  -\nabla\cdot(e^a\nabla u) = f \text{ in }\bbT^d = \R^d/\Z^d,
\end{equation}
for fixed $f\in H^{-1}(\bbT^d)$ with zero average in the unit cell.
We show that $\calG$ satisfies the assumptions of this paper, i.e.,
Assumption \ref{ass:G}.
We will consider the homogeneous norms on spaces of complex-valued functions
\begin{equation*}
  \| v \|^2_{\dotH^s(\bbT^d; \bbC)} = \sum_{k\in\Z^d} |k|^{2s} |\hat{v}_k|^2,
\end{equation*}
where $\hat{v}_k$ are the Fourier coefficients of $v$.
The proof is given in Section \ref{sec:diffusion-assumption-proof}.
\begin{proposition}
  \label{lemma:diffusion-assumption}
  Let $d\in \N$, $r>0$, $f\in H^{r-1}(\bbT^d; \bbR)$
  with\footnote{In the sense of distributions for $0<r<1$.}
  $\int_{(0,1)^d} f= 0$, $\calG : L^\infty(\bbT^d;
  \bbC)\to \{v \in H^1(\bbT^d; \bbC): \int_{(0,1)^d}v = 0\}$
  such that, for all $a\in L^\infty(\bbT^d; \bbR)$
  it holds in $H^{-1}(\bbT^d)$ that
  \begin{equation*}
    -\nabla\cdot(e^a\nabla\calG(a)) = f \;.
  \end{equation*}
Then,
$\calG$ satisfies Assumption \ref{ass:G} with $\calX = \dotH^{r+d/2}(\bbT^d;\bbR)$
and $\calY = \dotH^{r+1-\delta}(\bbT^d; \bbR)$, for any $0<\delta\leq r$.
\end{proposition}
Since the solution operator $\calG$ satisfies Assumption \ref{ass:G} with the spaces given in the
statement above, we can apply the approximation results of Section
\ref{sec:hermite-appx} to a parametrized version of the operators, defined as in
\eqref{eq:uy-def}.
\begin{theorem}
  \label{thm:diffusion-holomorphy}
  Let $d\in \N$, $r>0$, $f\in
  \dotH^{r-1}(\bbT^d;\bbC)$ with $\int_{(0,1)^d} f= 0$, $\calG :
  L^\infty(\bbT^d;\bbC)\to \{v \in H^1(\bbT^d;\bbC): \int_{(0,1)^d}v = 0\}$ such
  that, for all $a\in L^\infty(\bbT^d; \bbR)$ it holds in $H^{-1}(\bbT^d)$ that
  \begin{equation*}
    -\nabla\cdot(e^a\nabla\calG(a)) = f \;.
  \end{equation*}
  Let $(\psi_j)_{j\in\N}$ be a Riesz basis for $\dotH^{r+d/2}(\bbT^d; \bbR^d)$,
  $\bsb\in (0,\infty)^\infty\cap \ell^2(\N)$,
  and let $\delta>0$ be arbitrary, fixed.

  Then,
  there exists $\xi>0$ such that the function
  \begin{equation}
    \label{eq:u-diff}
    u :
    \begin{cases}
      \R^\infty \to \{v \in H^1(\bbT^d; \bbR):
      \int_{(0,1)^d}v = 0\} \\
      \bsy \mapsto \calG(\sum_{j\in\N} b_j y_j\psi_j)
    \end{cases}
  \end{equation}
  can be extended to a $(\bsb, \xi, 2,\dotH^{r+1-\delta}(\bbT^d; \bbC)
  )$-holomorphic function.
  Furthermore, if $\bsb\in \ell^p(\bbN)$ for $p\in(0,2)$, then for all $N\in\bbN$
  there exists $\Lambda_N\subset\calF$ with $|\Lambda_N| = N$ such that
  \begin{equation*}
  \left\|u(\bsy)-\sum_{\bsnu\in\Lambda_N}u_\bsnu H_\bsnu(\bsy)\right\|_{L^2_\gamma(\R^\infty,\dotH^{r+1-\delta}(\bbT^d; \bbR))}
  =
  \calO(N^{-\frac{1}{p}+\frac{1}{2}})\qquad\text{as }N\to\infty,
  \end{equation*}
  with $u_\bsnu$ defined as in \eqref{eq:unu}.
\end{theorem}
\begin{proof}
 The first part of the statement follows directly from Propositions \ref{prop:holomorphy} and
 \ref{lemma:diffusion-assumption}. The constant $\xi$ has to be taken
 \begin{equation*}
   \xi < \frac{\pi}{2C_{\mathrm{emb}}C_\psi},
 \end{equation*}
 where $C_{\mathrm{emb}}$ is the norm of the embedding operator of
 $\dotH^{r+d/2}(\bbT^d)$ into $L^\infty(\bbT^d)$ and $C_\psi$ is the Riesz basis
 constant of $(\psi_j)_{j\in\bbN}$ (see \eqref{eq:RszBas}).
 The application of Theorem \ref{thm:ubsnubound} concludes the proof.
\end{proof}

The following two remarks explain the improvement of our analysis compared to existing results.

\begin{remark}\label{remark:convergence-rates}[improved summability exponent]
  Let $\calX=L^2([0,2\pi])$ and let $\calY$ be a Hilbert space.
  Assume that $\mathcal{G}$ satisfies Assumption \ref{ass:G},
  and consider the orthonormal basis
  \begin{equation}\label{eq:xij}
    \psi_j(x):= \begin{cases}
                  \frac{1}{\sqrt{2\pi}} &\text{if $j=0$}\\
                  \frac{1}{\sqrt{\pi}}\cos(\frac{j}{2} x) &\text{if $j\ge 1$ is even}\\
                  \frac{1}{\sqrt{\pi}}\sin(\frac{j+1}{2} x) &\text{if $j\ge 1$ is odd},
                \end{cases}\qquad x\in [0,2\pi],~j\in\N_0
  \end{equation}
  of $\calX$. If $\bsb\in\ell^p(\N)$ for some $p\in (0,1)$,
  then by Proposition \ref{prop:holomorphy},
  \begin{equation*}
    u(\bsy)=\mathcal{G}\Big(\sum_{j\in\N}b_j y_j\psi_j\Big)
  \end{equation*}
  is $(\bsb,\theta/C_\pi,2,\delta,\calY_\C)$-holomorphic.
  On the other hand,
  by Remark \ref{remark:r-comparison} the map is
  also $(\bsb,\xi,1,\delta,\calY_\C)$-holomorphic for some $\xi>0$.

Applying \cite[Theorem 4.9]{DNSZ2023} for $(\bsb,\xi,1,\delta,\calY_\C)$-holomorphic functions,
we find that the Hermite coefficients satisfy $(\norm[\calY]{u_\bsnu})_{\bsnu\in\calF}\in\ell^{2p/(2-p)}(\calF)$.
On the other hand,
Theorem \ref{thm:bdXSum} gives $(\|{u_\bsnu}\|_{\calY})_{\bsnu\in\calF}\in\ell^{p}(\calF)$.
Due to $p<2p/(2-p)$,
Theorem \ref{thm:bdXSum}
provides sharper summability estimates in this case
and, consequently,
leads to higher approximation rates for best $n$-term approximation.
\end{remark}

\begin{remark}[improved summability range]
\label{rmk:CmpBCDM2016}
  The results in this manuscript additionally ``bridge the gap''
  between the orders of convergence obtained in \cite{BCDM2016} for globally
  supported bases, see \cite[Corollary 6.3]{BCDM2016},
  and the technique of holomorphic extension used here and in \cite{DNSZ2023}
  when applied to the low Sobolev-regularity setting.

  Let the trigonometric basis $\psi_j$ of $L^2([0,2\pi])$ be as in \eqref{eq:xij}.
  Let also $\eps>0$ be arbitrary
  and introduce the $H^{\eps+1/2}$-orthonormal basis $(\tilde\psi_j)_{j\in\bbN_0}$
  such that
  \begin{equation*}
    \tilde \psi_j \propto j^{-\eps-1/2}\psi_j, \qquad j\in\bbN_0.
  \end{equation*}
Let $u$ be defined as in \eqref{eq:u-diff}, with the above $\tilde\psi_j$ and with
  $\bsb\in \ell^p(\bbN_0)$. Then,
  Propositions
  \ref{prop:holomorphy} and \ref{lemma:diffusion-assumption} (with $r=\eps$),
  and Theorem \ref{thm:bdXSum} imply
  that $\| u_{\bsnu}\|_{\dotH^1}\in\ell^p(\calF)$.
  On the other hand, $\bsb\in \ell^p(\bbN_0)$ also implies
  $(j^{\eps+1/2}\|b_j\psi_j\|_{L^\infty} )_{j\in\bbN_0} \in \ell^p(\bbN_0)$, which
  is a sufficient condition for using \cite[Corollary 6.2]{BCDM2016} to
  show the $\ell^p(\calF)$ summability of $\|u_\bsnu\|_{\dotH^1}$,
  as outlined in the text after that statement.
\end{remark}

\section{Conclusions and future work}
\label{sec:Con}
We investigated the sparsity of coefficient sequences of Wiener-Hermite
polynomial chaos expansions of holomorphic maps $\calG$ between Hilbert spaces.
To this end, we used and extended a quantified notion of holomorphic dependence recently
developed in \cite{DNSZ2023}. We estimated the coefficient sequence summability
exponent for the images under $\calG$, subject to suitable encoder maps. The
setting adopted is abstract, and covers a range of maps. In particular, we
verify our abstract hypotheses for the coefficient-to-solution maps of a linear,
divergence-form elliptic PDE whose coefficients are modeled as a log-Gaussian
random field. This model has been investigated in recent years, e.g. in
\cite{BCDM2016}, by means of real-variable arguments. The arguments developed here extend the range of summability exponents as compared to \cite{DNSZ2023}, and recover, as a particular case, for globally supported, orthonormal encoding functions, the result from that reference.

The present results will be employed in \cite{MSZII-25} to infer expression rate bounds for certain neural operator networks and for finitely truncated, Wiener-Hermite polynomial chaos expansions to approximate holomorphic maps $\calG$ between separable function spaces.

\appendix

\section{Results on $(b,\xi,2,\delta, X)$-holomorphic functions}
\label{sec:AppHol}
Recall (see the proof of Theorem \ref{thm:bdXSum}) that, when $u$ is $(\bsb,\xi,r, \delta,X)$-holomorphic, then there exist,
for all $N\in\bbN$, functions $u_N:\R^N\to X$ satisfying items
\ref{item:hol}--\ref{item:vN} of
Definition~\ref{def:bdXHol}. We can expand those functions in the PC expansion
\begin{equation*}
  u_N=\sum_{\bsnu\in\N_0^N} u_{N,\bsnu} H_\bsnu,
\end{equation*}
where
\begin{equation}
  \label{eq:uNnu}
  u_{N,\bsnu}=\int_{\R^N} u_N(\bsy) H_\bsnu(\bsy) \rd \gamma_N(\bsy).
\end{equation}
In the following statement, we bound the derivatives of $u_N$, both in a
pointwise and $L^2_{\gamma_N}$ sense. The proof follows closely that of
\cite[Lemma 4.6]{DNSZ2023}.
\begin{lemma}\label{holo-lem1}
  Let $\bsb\in(0, \infty)^\infty$, $\xi>0$, $\delta>0$, and let $X$ be a
  complex-valued, separable Hilbert space.
  Let $u$ be $(\bsb,\xi,2,\delta, X)$-holomorphic.
  Let $N\in\N$ and let $u_N$, $\varphi_N$ be functions satisfying items
  \ref{item:hol}--\ref{item:vN} of Definition~\ref{def:bdXHol}.

  Then, denoting $\bsb_N=(b_j)_{j=1}^N$, for every
  $\bsnu\in\N_0^N$ it holds that
  \begin{equation}
    \label{eq:partialuN-pointwise}
    \|\partial^{\bsnu}u_N(\bsy)\|_X
    \leq
    \frac{\bsnu!|\bsnu|^{|\bsnu|/2}\bsb_N^\bsnu}{\kappa^{|\bsnu|}\bsnu^{\bsnu/2} }
    \varphi_N(\bsy)\qquad\forall\bsy\in\R^N.
  \end{equation}
  In addition,
  \begin{equation}
    \label{eq:partialuN-L2}
    \| \partial^{\bsnu} u_N \|^2_{L^2_{\gamma_N}(\R^N, X)} \leq
    \frac{|\bsnu|^{|\bsnu|}}{\bsnu^\bsnu}\left( \frac{\bsnu!\bsb_N^\bsnu }{\kappa^{|\bsnu|}} \right)^2 \| \varphi_N \|^2_{L^2_{\gamma_N}(\R^N)}.
  \end{equation}
\end{lemma}
\begin{proof}
  Fix $\kappa$ such that $0< \kappa<\xi$. The case $\bsnu=(0,\dots,0)\in\N_0^N$ is trivial. We thus fix a nonzero $\bsnu\in \N_0^N$ and
  choose $\bsvarrho=(\varrho_j)_{j=1}^N$
  with
  \begin{equation*}
    \varrho_j^2 =
    \begin{dcases}
      \dfrac{\kappa^2}{b_j^2} \dfrac{\nu_j}{|\bsnu|}  &\text{if } j\in \supp(\bsnu), \\
      \dfrac{\xi^2 -\kappa^2}{b_j^2 N} &\text{if } j\not \in \supp(\bsnu).
    \end{dcases}
  \end{equation*}
  Then
  \begin{equation*}
    \sum_{j=1}^N
    \varrho_j^2b_j^2=\kappa^2 \sum_{j\in \supp(\bsnu)} \frac{\nu_j}{|\bsnu|}
    +\sum_{j\not\in \supp(\bsnu)}\frac{\xi^2-\kappa^2}{N} \leq \kappa^2 + \frac{N-1}{N}(\xi^2-\kappa^2)< \xi^2.
  \end{equation*}
  Hence there exists a $\bsdelta\in \R^N$ with $\delta_j>0$ such that $\bsvarrho+\bsdelta$ is $(\bsb,\xi, 2)$-admissible, i.e., there exists a
  holomorphic extension $u_N:\cS_{\bsvarrho+\bsdelta}^\bbC\to X$ as in
  Definition~\ref{def:bdXHol} \ref{item:hol}-\ref{item:varphi}.
  By Cauchy's integral formula, for all $\bsy\in \R^N$,
    \begin{equation*}
    \partial^{\bsnu} u_N(\bsy) = \frac{\bsnu!}{(2\pi \ii)^N} \int_{\partial\cC_\bsvarrho (\bsy)} \frac{u_N(\bsz)}{\prod_{j\in \{1, \dots, N\}} (z_j - y_j)^{\nu_j+1}} \prod_{j\in \{1, \dots, N\}} dz_j,
  \end{equation*}
  with
  \begin{equation*}
    \partial \cC_{\bsvarrho}(\bsy) =  \left\{ \bsz\in \C^N: |z_j - y_j | = \rho_j, \, j\in \{1, \dots, N\}\right\}.
  \end{equation*}
  Hence,
  \begin{equation*}
    \| \partial^{\bsnu} u_N(\bsy) \|_X \leq \frac{\bsnu!}{\bsrho^{\bsnu}} \sup_{\bsz \in \calB(\bsvarrho)} \| u_N (\bsy + \bsz) \|_{X}.
  \end{equation*}
  Substituting the values for $\bsvarrho$, we obtain
  \begin{equation*}
    \| \partial^{\bsnu} u_N(\bsy) \|_X \leq \frac{\bsnu! |\bsnu|^{|\bsnu|/2} \bsb_N^\bsnu}{\kappa^{|\bsnu|} \bsnu^{\bsnu/2}}\sup_{\bsz \in \calB(\bsvarrho)} \| u_N (\bsy + \bsz) \|_{X} .
  \end{equation*}
  Using Definition~\ref{def:bdXHol} \ref{item:varphi} concludes the proof of
  \eqref{eq:partialuN-pointwise}. Then, \eqref{eq:partialuN-L2} follows from
  squaring and integrating.
\end{proof}
In the following, we will use, for $M\in \N$ and for $\bsvarrho_N \in
(0,\infty)^N$, the quantities
\begin{equation} \label{eq:betaN}
	\beta_\bsnu(M,\bsvarrho_N) \coloneqq
	\sum_{\bsnu' \in \calF: \ \|\bsnu'\|_{\ell^\infty}\leq M} \binom{\bsnu}{\bsnu'}
	\bsvarrho_N^{2\bsnu'} = \prod_{j=1}^N\Bigg(\sum_{\ell=0}^{M}\binom{\nu_j}{\ell}\varrho_j^{2\ell}\Bigg),
	\quad \bsnu \in \calF.
\end{equation}
Note that $\beta_\bsnu(M,\bsvarrho_N)$ is obtained from the definition
\eqref{eq:beta} of $\beta_\bsnu$ by padding $\bsvarrho_N$ with zeros to an infinite sequence.
\begin{lemma}\label{holo-lem2}
Let $\bsb\in(0, \infty)^\infty$, $\xi>0$, $\delta>0$, and let $X$ be a
  complex-valued, separable Hilbert space.
  Let $u$ be $(\bsb,\xi,2,\delta,X)$-holomorphic with $u_{N, \bsnu}$ defined as
  in \eqref{eq:uNnu}. Let $N\in\N$ and $\bsvarrho_N= (\varrho_j)_{j=1}^N\in [0,\infty)^N$.

  Then, for any $M\in\N$,
  \begin{equation} \label{general}
    \sum_{\bsnu\in
      \N_0^N}\beta_\bsnu(M, \bsvarrho_N )\|u_{N,\bsnu}\|_{X}^2
    =
    \sum_{\bsnu\in\N_0^N:\|\bsnu\|_{\ell^\infty}\leq M}
    \frac{\bsvarrho_N^{2\bsnu}}{\bsnu!}
    \| \partial^\bsnu u_N\|_{L^2_{\gamma_N}(\R^N, X)}^2.
  \end{equation}
\end{lemma}
\begin{proof}
The proof is the same as that of \cite[Lemma 4.7]{DNSZ2023} or \cite[Theorem 3.3]{BCDM2016}.
\end{proof}
\begin{proposition} \label{prop:bdHolSum}
Let $p\in (0,2)$, $\bsb\in(0, \infty)^\infty\cap \ell^p(\bbN)$, $\xi>0$, $\delta>0$, $M\in\N$
 with $M\geq 2/p$, and let $X$ be a
  complex-valued, separable Hilbert space.
Let $u$ be
$(\bsb,\xi,2, \delta,X)$-holomorphic. Let $u_N$ be functions satisfying items
  \ref{item:hol}--\ref{item:vN} of
  Definition~\ref{def:bdXHol} and let $u_{N, \bsnu}$ be defined as in \eqref{eq:uNnu}.
Let $\bsvarrho \in (0, \infty)^\infty$ be defined as in \eqref{eq:varrho_j}
and  $\bsvarrho_N=(\varrho_j)_{j=1}^N$.  Then, for all $N\in\N$,
\begin{equation}\label{eq:generalN}
  \sum_{\bsnu\in \N_0^N}\beta_\bsnu(M,\bsvarrho_N)\norm[X]{u_{N,\bsnu}}^2\le
  2\delta^2, \qquad
  \norm[\ell^{p/(2-p)}(\N_0^N)]{\beta_\bsnu(M,\bsvarrho_N)^{-1}} \le C'%
\end{equation}
for a finite constant $C'$ depending on $\bsb$, $\xi$, and $M$,
but independent of $\delta$ and $N\in\N$.
Furthermore, for every $N\in\N$ and every $s>0$,
$(\norm[X]{u_{N,\bsnu}})_{\bsnu\in\N_0^N}\in \ell^{s}(\N_0^N)$.
If in addition $s\geq p$, then the norm is independent of $N$, i.e.,
there exists a constant $C>0$ such that, for all $N\in \N$,
\begin{equation*}
\big\|(\norm[X]{u_{N,\bsnu}})_{\bsnu\in\N_0^N}\big\|_{\ell^{s}(\N_0^N)}\leq C <
\infty\;.
\end{equation*}
\end{proposition}
\begin{proof}
  The proof follows that of \cite[Theorem 4.8]{DNSZ2023}.
  We have, by definition,
  $$\sum_{j\in\N}\varrho^2_jb^2_j=\frac{\xi^2}{8e M! }<\infty,$$ and
  $(\varrho_j^{-1})_{j\in\N}\in\ell^{2p/(2-p)}(\N)$.  Set
  $\kappa:=\xi/2\in (0,\xi)$.
  Using \eqref{general}, \eqref{eq:partialuN-L2}, and item \ref{item:varphi} of Definition
  \ref{def:bdXHol} we obtain
  \begin{align*}
    \sum_{\bsnu\in \N_0^N}\beta_\bsnu(M,\bsvarrho_N)\norm[X]{u_{N,\bsnu}}^2
    &\leq \delta^{2}
      \sum_{\set{\bsnu\in\N_0^N}{\norm[\ell^\infty]{\bsnu}\le M}} \frac{|\bsnu|^{|\bsnu|} \bsnu!}{\bsnu^{\bsnu}}\left(\frac{\bsvarrho_N^{\bsnu} \bsb_N^{\bsnu}}{\kappa^{|\bsnu|}}\right)^2\\
    &\leq \delta^{2} \sum_{\set{\bsnu\in\N_0^N}{\norm[\ell^\infty]{\bsnu}\le M}} \frac{|\bsnu|^{|\bsnu|} \bsnu!}{\bsnu^{\bsnu}}\left( \frac{2}{\xi} \right)^{2|\bsnu|}\left( \frac{\xi^2 }{4}\frac{\bsb_N^p}{ M! 2e\|\bsb\|^p_{\ell^p}}\right)^{\bsnu}\\
    &\leq \delta^{2} \sum_{\set{\bsnu\in\N_0^N}{\norm[\ell^\infty]{\bsnu}\le M}} \frac{|\bsnu|^{|\bsnu|}}{\bsnu^{\bsnu}} \left(\frac{ \bsb_N^{p}}{2e\|\bsb\|^{p}_{\ell^p}}  \right)^{\bsnu}\\
    & \eqqcolon \delta^2 C_M(\bsb_N).
  \end{align*}
Here we have used the inequality $\bsnu! \leq M!^{|\bsnu|}$,
valid for $\|\bsnu\|_{\ell^\infty}\leq
  M$.
Let $\tilde{\bsb} = \bsb^p / (2\|\bsb\|^p_{\ell^p})$, so that $\|\tilde\bsb\|_{\ell^1(\N)}=1/2$.
Using the Stirling inequality \cite[Equations (3.124), (3.127)]{CoDe2015}, we have
  \begin{equation}
    \label{eq:Cbsb}
    C_M(\bsb_N)\leq
    \sum_{\bsnu \in \calF} \frac{|\bsnu|^{|\bsnu} }{\bsnu^\bsnu} \left(\frac{\tilde{\bsb}}{e}  \right)^{\bsnu}\leq
    \sum_{\bsnu \in \calF} \frac{|\bsnu|!}{\bsnu!}\tilde{\bsb}^{\bsnu} = \frac{1}{1-\|\tilde\bsb\|_{\ell^1(\N)}} = 2.
  \end{equation}

  We temporarily extend $\bsvarrho_N$ to an infinite sequence,
  in order to derive the summability of
  the weight sequence $\beta_\bsnu(M,\bsvarrho_N)^{-1}$
  through \cite[Lemma 3.11]{DNSZ2023}.
  To do so, introduce the vector $\tilde\bsvarrho_N \in (0, \infty)^\infty$
  such that
  $\tilde \varrho_{N,j}:=\varrho_j$ if $j\le N$ and
  $\tilde \varrho_{N,j}:=\exp(j)$ otherwise.
  For any $q>0$ we then
  have $(\tilde \varrho_{N,j}^{-1})_{j\in\N}\in\ell^{q}(\N)$ and by
  \cite[Lemma 3.11]{DNSZ2023} this implies
  $$(\beta_\bsnu(M,\tilde\bsvarrho_N)^{-1})_{\bsnu\in\calF}\in\ell^{q/2}(\calF)$$
  as long as $M>2/q$.
  We observe that, if $\bsnu\in\calF$ with
  $\supp\bsnu\subseteq \{1,\dots,N\}$, then
  $\beta_{\bsnu}(M,\tilde\bsvarrho_N)=\beta_{\bsnu}(M,\bsvarrho_N)$.
  It follows that
  \begin{equation}
    \label{eq:beta-N-q/2}
    (\beta_{\bsnu}(M,\bsvarrho_N)^{-1})_{\bsnu\in\N_0^N}\in \ell^{q/2}(\N_0^N)
  \end{equation}
  for any $q>0$.

  Consider now a fixed $s \in (0,2)$
  and let $M\in\N$
  such that $M>(2-s)/s$.
  H\"older's inequality gives
  \begin{align*}
    \sum_{\bsnu\in\N_0^N}\norm[X]{u_{N,\bsnu}}^{s}
    &= \sum_{\bsnu\in\N_0^N}\norm[X]{u_{N,\bsnu}}^{s}\beta_{\bsnu}(M,\bsvarrho_N)^{\frac s 2}\beta_{\bsnu}(M,\bsvarrho_N)^{-\frac s 2}\nonumber\\
    &\le \Bigg(\sum_{\bsnu\in\N_0^N}\norm[X]{u_{N,\bsnu}}^{2}\beta_{\bsnu}(M,\bsvarrho_N) \Bigg)^{\frac s 2}\Bigg(\sum_{\bsnu\in\N_0^N}\beta_{\bsnu}(M,\bsvarrho_N)^{\frac{s}{s-2}} \Bigg)^{\frac{2-s}{2}}
  \end{align*}
  Using \eqref{eq:beta-N-q/2}, we can see that the second term in the
  multiplication above is bounded for all $s\in (0,2)$ and for all $N\in\bbN$.
  When $s>p$, in addition, we exploit \cite[Lemma 3.11]{DNSZ2023} which implies
there exists a constant $C'(\bsb,\xi, M)$ such that for all $N\in \N$,
  $$
  \big\|(\beta_\bsnu(M,\bsvarrho_N)^{-1})_{\bsnu\in \N^N_0}\big\|_{\ell^{s/(2-s)}(\N_0^N)}\leq
  \big\|(\beta_\bsnu(M,\bsvarrho_N)^{-1})_{\bsnu\in \calF}\big\|_{\ell^{p/(2-p)}(\N_0^N)}
  \leq C'(\bsb,\xi, M) < \infty \;.
  $$
This completes the proof.
\end{proof}

\section{Proof of Lemma \ref{lemma:cnu}}\label{sec:proofcnu}
The following proof corresponds, up to minor modifications and simplifications, to the proof of \cite[Lemma 6.5 and 6.6]{DNSZ2023}.

\begin{proof}[Proof of Lemma \ref{lemma:cnu}]
  \textbf{Step 1.} Fix $\bsnu\in\calF$ and let $j\in\N$.
  With $s:=\min\{M,\nu_j\}\le \nu_j$
  it then holds
  \begin{equation*}
    \binom{\nu_j}{s}
    =
    \frac{\nu_j!}{(\nu_j-s)! s!}
    \ge
    \frac{1}{s!} (\nu_j-s+1)^s\ge \nu_j^s \frac{1}{s! s^s}
    \ge
    \nu_j^s \frac{1}{M! M^M}= \nu_j^{\min\{\nu_j,M\}} \frac{1}{M!M^M}
    \ge
    \nu_j^{M} \frac{1}{M!M^{2M}}.
  \end{equation*}
  Furthermore,
  if $j\in\supp(\bsnu)$, then due to
  $s=\min\{\nu_j,M\}\ge 1$, with
  $\varrho_0:=\min \{1,\min_{j\in\N}\varrho_j\}$ we have
$$
\varrho_0^{2M} \leq \min\{ 1,\varrho_j\}^{2M} \leq \varrho_j^{2(s-1)}.
$$
Thus
$$\varrho_j^{\min\{\nu_j,M\}}\ge \varrho_0^{2M} \varrho_j^{2}$$ for all $j\in\N$.
In all, we conclude
\begin{equation}\label{eq:estb}
  \beta_\bsnu(M,\bsvarrho) = \prod_{j\in\N} \left(\sum_{l=0}^M
    \binom{\nu_j}{l}\varrho_j^{2l} \right) \ge
  \prod_{j\in\supp(\bsnu)}\binom{\nu_j}{\min\{\nu_j,M\}}
  \varrho_j^{2\min\{\nu_j,M\}} \ge \prod_{j\in\supp(\bsnu)}
  \frac{\varrho_0^{2M}}{M!M^{2M}} \varrho_j^{2}\nu_j^{M}.
\end{equation}
Since $\bsnu\in\calF$ was arbitrary, this estimate holds for all $\bsnu\in\calF$.

\textbf{Step 2.}
Denote $\hat\varrho_j:=\max\{1,K\varrho_j\}$
and note that for all $\bsnu\in\cF$
\begin{equation}\label{eq:estc}
  c_{\bsnu} \le
  \prod_{j\in\supp(\bsnu)}
  \hat\varrho_j^{2} \nu_j^{M}.
\end{equation}
With $\varrho_0:=\min\{1,\min_{j\in\N}\varrho_j\}$ set
\begin{equation*}
  K:=\left(\frac{\varrho_0^{2M}}{M!M^{2M}}\right)^{1/2},\qquad \tilde\varrho_j= K \varrho_j
\end{equation*}
such that for all $j\in\N$
\begin{equation*}
  K\varrho_j=\tilde\varrho_j = \hat\varrho_j\begin{cases}
    1 &\text{if }K\varrho_j\ge 1,\\
    K\varrho_j  &\text{if }K\varrho_j<1.
  \end{cases}
\end{equation*}
Let
\begin{equation*}
  C_\beta:= \prod_{\set{j\in\N}{K\varrho_j<1}} (K\varrho_j)^{2}.
\end{equation*}
This product is over a finite number of indices,
since $\varrho_j\to\infty$ as $j\to\infty$.

Then, for any $\bsnu\in\cF$,
\begin{equation*}
  \prod_{j\in\supp(\bsnu)} \tilde\varrho_j \ge C_\beta^{\frac{1}{2}} \prod_{j\in\supp(\bsnu)} \hat\varrho_j.
\end{equation*}
With \eqref{eq:estb} and \eqref{eq:estc} we thus obtain for every $\bsnu\in\cF$,
\begin{align*}
  \beta_\bsnu(M,\bsvarrho) &\ge \prod_{j\in\supp(\bsnu)} (K\varrho_j)^{2}\nu_j^M
 = \prod_{j\in\supp(\bsnu)} \left( \tilde \varrho_j\right)^{2}\nu_j^M \notag \\
                         &\ge C_\beta \prod_{j\in\supp(\bsnu)} (\hat\varrho_j)^{2}\nu_j^M \ge C_\beta c_{\bsnu}.
\end{align*}
This shows \eqref{eq:cnu-betanu}.

{\bf Step 3.}
By assumption, $\hat\varrho_j=\max\{1,K\varrho_j\}$
satisfies $(\hat\varrho_j^{-1})_{j\in\N}\in\ell^q(\N)$.
By definition of $c_{\bsnu}$, factorizing, we get
  \begin{equation*}
    \begin{split}
      \sum_{\bsnu\in\cF} c_{\bsnu}^{-\frac{q}{2}} & =
      \sum_{\bsnu\in\cF} \prod_{j\in\supp(\bsnu)}
      \left(\hat\varrho_j^{2} \nu_j^{M}\right)^{-\frac{q}{2}}
      \leq \prod_{j\in\N} \left( \hat\varrho_j^{-q}
        \sum_{n\in\N} n^{\frac{-q M}{2}} \right).
    \end{split}
  \end{equation*}
  The sum over $n$ equals some finite constant $C$ since by assumption
  $qM/2>1$.
  Using the inequality $\log(1+x)\le x$ for all $x>0$, we get
  \begin{equation*}
    \sum_{\bsnu\in\cF} c_{\bsnu}^{-\frac{q}{2}}\le \prod_{j\in\N}
    \left( 1 + C \hat\varrho_j^{-q}\right)
    = \exp\left(\sum_{j\in\N}\log (1+C \hat\varrho_j^{-q}) \right)
    \le \exp\left(\sum_{j\in\N} C \hat\varrho_j^{-q}\right),
  \end{equation*}
  which is finite since $(\hat\varrho_j^{-1})\in\ell^q(\N)$.
\end{proof}

\section{Summability of products of square exponentials}
\label{sec:SumSqExp}
We prove a technical lemma which is needed for the proof of Proposition \ref{prop:holomorphy}.
\begin{lemma}
  \label{lemma:expL2}
Let $\bsb\in (0,\infty)^\infty \cap \ell^2(\N)$, $\alpha \in (0, \infty)$, and $\tau\in (0,2)$.

Then, there exists $\delta>0$
such that for all $N\in \N$
  \begin{equation*}
   \int_\R^N \exp \Big( \alpha \Big(\sum_{j=1}^N b_j^2y_j^2\Big)^{\tau/2}\Big) \gamma_N(d(y_j)_{j=1}^N)
   \leq \delta.
  \end{equation*}
  Furthermore,
  $\bsy \mapsto \exp \Big( \alpha \Big(\sum_{j\in \N} b_j^2y_j^2\Big)^{\tau/2}\Big) \in L^2_\gamma(\R^\infty)$.
\end{lemma}
\begin{proof}
  Let $\tilde{\alpha} > 0$ be small
  enough so that $\tilde{\alpha}\max_{j\in \N}b_j ^2 <1/2$ and let
  $C_{\bsb,\alpha, \tau}>0$ be a constant such that
  \begin{equation*}
    e^{\alpha x^{\tau/2}} \leq C_{\bsb, \alpha, \tau} e^{\tilde{\alpha} x}, \qquad \forall x\in\R_+.
  \end{equation*}
  We have, with $\bsy_N = (y_1, \dots, y_N)$,
  \begin{align*}
    \int_{\R^N} \exp\Big(\alpha\Big(\sum_{j=1}^N b_j^2y_j^2\Big)^{\tau/2} \Big) \gamma_N(d\bsy_N)
    &\leq C_{\bsb,\alpha, \tau} \int_{\R^N} \exp\Big(\tilde\alpha  \sum_{j=1}^Nb_j^2y_j^2\Big)\gamma_N(d\bsy_N)\\
    &\leq C_{\bsb,\alpha, \tau} \prod_{j=1}^N \left( \frac{1}{1-2\tilde\alpha b_j^2} \right)^{1/2}.
  \end{align*}
  To conclude the proof, it is then sufficient to show that
  \begin{equation}
    \label{eq:prod-bounded}
    \prod_{j=1}^\infty \left( \frac{1}{1-2\tilde\alpha b_j^2} \right) < \infty.
  \end{equation}
  To this end we first remark that
  \begin{equation*}
    \prod_{j=1}^\infty \left( \frac{1}{1-2\tilde\alpha b_j^2} \right)
     =
    \exp\left( \sum_{j\in \N}\log\frac{1}{1 - 2\tilde\alpha b_j^2}  \right)
     =
    \exp\left( \sum_{j\in \N}\log\left( 1 + \frac{2\tilde\alpha b_j^2}{1 - 2\tilde\alpha b_j^2} \right)  \right).
  \end{equation*}
  Since $b_j^2\to 0$ as $j\to\infty$, a sufficient condition for
  \eqref{eq:prod-bounded} is the convergence of the series $\sum_{j\in\N}
  \frac{b_j^2}{1-2\tilde\alpha b_j^2}$, which follows from
  $\bsb\in \ell^2(\N)$.
\end{proof}
\section{Regularity of the solution operator of \eqref{eq:diffusion-torus}}
\label{sec:RegDiffOp}
We will use, for the next statements,
the Hölder spaces $C^{k,\alpha}(\bbT^d; \bbC)$, defined for
$k\in\N_0$ and $\alpha\in (0,1)$, with norm
\begin{equation*}
  \|v\|_{C^{k, \alpha}(\bbT^d; \bbC)} = \sum_{\beta\in \N^d_0 : |\beta|\leq k} \|\partial^\beta v\|_{L^\infty(\bbT^d)} + \sup_{\beta\in\N_0^d: |\beta|=k}\left[ \partial^\beta v \right]_\alpha,
\end{equation*}
where
\begin{equation*}
  \left[ w \right]_\alpha = \sup_{x, y\in \bbT^d }\frac{|w(y)-w(x)|}{|y-x|^\alpha}.
\end{equation*}
When we write $C^{k, \alpha}(\bbT^d; \bbC)$ for $\alpha=0$, we imply $C^{k,
  0}(\bbT^d; \bbC) = C^k(\bbT^d; \bbC)$.

The domain of most spaces in the following is the $d$-dimensional torus,
and this argument will be mostly omitted to lighten notation.
In addition, \emph{all functions spaces which appear in the following
will be understood as spaces of complex-valued functions},
and this will also not be indicated explicitly.

\subsection{Bounds in $\dotH^{1+s}$ with $s\in (0,1)$ for solutions of the Poisson equation}
\label{sec:Ht}
In this section, we derive a data dependent bound for the solution $u$ of the
problem
\begin{equation}
  \label{eq:generic-poisson}
  -\nabla\cdot(b\nabla u) = f\text{ in }\bbT^d, \qquad \int u  = 0
\end{equation}
where $b\in L^\infty(\bbT^d ; \bbC)$, with real part bounded from below almost
everywhere by a positive constant, and $f\in \dotH^{-1}(\bbT^d, \bbC)$.
We denote by $\calS$ the operator $\calS: (b, f)\mapsto u$.
\begin{lemma}
  \label{lemma:H1diff-poisson}
  Let $f, f_\delta\in \dotH^{-1}(\bbT^d; \bbC)$ and let $b, b_\delta \in
  L^\infty(\bbT^d; \bbC)$ with $\Re (b(x))\geq
  b_{\min{}}>0$ and $\Re (b_\delta(x))\geq b_{\delta,\min{}}>0$ almost everywhere. Then,
  \begin{equation*}
    \| \calS(b, f) - \calS(b_\delta, f_\delta) \|_{\dotH^1} \leq \frac{1}{b_{\delta, \min{}}}\left( \|f-f
    _\delta \|_{\dotH^{-1}} + \frac{1}{b_{\min{}}}\|f\|_{\dotH^{-1}}\|b-b_\delta\|_{L^{\infty}}\right)
  \end{equation*}
\end{lemma}
\begin{proof}
  The statement follows from the following series of inequalities, where we
  denote $u_\delta = \calS(b_\delta, f_\delta)$ and $u = \calS(b, u)$:
  \begin{align*}
    b_{\delta, \min{}}\| u - u_\delta \|_{\dotH^1}
    &\leq \|\nabla \cdot(b_\delta \nabla u_\delta)  - \nabla \cdot(b_\delta \nabla u)\|_{\dotH^{-1}} \\
    & \leq \| f_\delta - \nabla \cdot(b_\delta \nabla u) \|_{\dotH^{-1}}\\
    & \leq \| f_\delta - f \|_{\dotH^{-1}} + \|\nabla \cdot((b-b_\delta) \nabla u) \|_{\dotH^{-1}}\\
    & \leq \| f_\delta - f \|_{\dotH^{-1}} + \|b - b_\delta\|_{L^\infty}\|\nabla u \|_{L^2}\\
    & \leq \| f_\delta - f \|_{\dotH^{-1}} + \frac{1}{b_{\min{}}}\|b - b_\delta\|_{L^\infty}\|f \|_{\dotH^{-1}}.
  \end{align*}
\end{proof}
\begin{proposition}
  \label{prop:Hs-poisson}
Let $u$ be the solution to \eqref{eq:generic-poisson} with
$f\in \dotH^{-1+s}(\bbT^d;\bbC)$
and
$b\in C^{0,s}(\bbT^d;\bbC)$, for $s\in[0,1]$, and $\Re(b(x))\geq b_{\min{}}>0$ almost everywhere.
Then, there exists $C>0$ independent of $f$ and
$b$ such that
\begin{equation*}
  \| u\|_{\dotH^{1+s}} \leq \frac{C}{b_{ \min{} }} \left(   \|f\|_{\dotH^{-1+s}} + \frac{1}{b_{\min{}}}\|f\|_{\dotH^{-1}}\|b\|_{C^{0,s}} \right).
\end{equation*}
\end{proposition}
\begin{proof}
  We use the fact that
  \begin{equation*}
    \| u\|_{\dotH^{1+s}} = \sup_{t>0} t^{-s}K(u, t, \dotH^1, \dotH^2),
  \end{equation*}
  with
  \begin{equation*}
    K(u, t, \dotH^1, \dotH^2) = \inf_{v_\delta\in \dotH^2} \|u -v_\delta\|_{\dotH^1} + t\|v_\delta\|_{\dotH^2}.
  \end{equation*}
  There exists $C_1\geq 1$ such that, for all $\delta\in (0,1)$, there exist $f_\delta$ and $b_\delta$ such that
  \begin{align*}
    \|b-b_\delta \|_{L^\infty} \leq C_1 \delta^{s}\|b\|_{C^{0,s}} \qquad \| b_\delta\|_{W^{1, \infty}}\leq C_1 \delta^{-1+s} \|b\|_{C^{0,s}}\\
    \|f-f_\delta \|_{\dotH^{-1}} \leq C_1 \delta^{s}\|f\|_{\dotH^{-1+s}} \qquad \| f_\delta\|_{L^2}\leq C_1 \delta^{-1+s} \|f\|_{\dotH^{-1+s}}\qquad \|f_\delta\|_{\dotH^{-1}}\leq C_1\|f\|_{\dotH^{-1}}.
  \end{align*}
  Let now $\delta^{s}\leq b_{\min{}}/(2C_1 \|b\|_{C^{0,s}}) \eqqcolon \delta_0^s$, so that $\Re(b_\delta(x))\geq
  b_{\delta, \min{}} \geq b_{\min{}}/2$ almost everywhere. We denote $u_\delta =
  \calS(b_\delta, f_\delta)$, so that
  \begin{align*}
    \| u_\delta \|_{\dotH^2}
    &\leq \frac{2}{b_{\min{}}} \left( \|f_\delta\|_{L^2} + \frac{2\|b_\delta\|_{W^{1, \infty}}}{ b_{\min{}}} \|f_\delta\|_{\dotH^{-1}}\right)\\
    & \leq\frac{2 C_1}{b_{\min{}}} \delta^{-1+s}\left( \|f\|_{\dotH^{-1+s}} + \frac{2C_1\|b\|_{C^{0,s}}}{ b_{\min{}}} \|f\|_{\dotH^{-1}}\right) .
  \end{align*}
  It follows from Lemma
  \ref{lemma:H1diff-poisson} and the inequalities above that, for sufficiently
  small $\delta$,
  \begin{equation}
    \label{eq:bound-K}
  \begin{aligned}
    &K(u, t, \dotH^1, \dotH^2) \\
    &\quad\leq \| u - u_\delta \|_{\dotH^1} + t\|u_\delta\|_{\dotH^2}\\
    &\quad\leq  \frac{2 C_1\delta^s}{b_{\min{}}} \left(   \|f\|_{\dotH^{-1+s}} + \frac{2}{b_{\min{}}}\|f\|_{\dotH^{-1}}\|b\|_{C^{0,s}}  + t \delta^{-1} \left(\|f\|_{\dotH^{-1+s}}  + \frac{2C_1}{b_{\min}}\| b\|_{C^{0,s}}\|f\|_{\dotH^{-1}}\right)   \right)
    \\ &\quad\leq
\frac{2 C_1^2\delta^s}{b_{\min{}}} \left( 1+t\delta^{-1} \right)\left(   \|f\|_{\dotH^{-1+s}} + \frac{2}{b_{\min{}}}\|f\|_{\dotH^{-1}}\|b\|_{C^{0,s}}\right).
  \end{aligned}
  \end{equation}
  We conclude by remarking that, given $t_0>0$,
  \begin{equation*}
\sup_{t>0} t^{-s}K(u, t, \dotH^1, \dotH^2)\leq \max\left(\sup_{t\in(0, t_0)} t^{-s}K(u,t, \dotH^1_0, \dotH^2),  t_0^{-s}\|u\|_{\dotH^1_0}\right),
  \end{equation*}
  hence, choosing $\delta = t$ in \eqref{eq:bound-K} and $t_0 = \delta_0$, we obtain
  \begin{align*}
    \| u\|_{\dotH^{1+s}}
    &\leq \max\left(\frac{4C_1^2}{b_{\min{}}}\left(   \|f\|_{\dotH^{-1+s}} + \frac{2}{b_{\min{}}}\|f\|_{\dotH^{-1}}\|b\|_{C^{0,s}} \right),  \frac{2C_1\|b\|_{C^{0,s}}\|f\|_{\dotH^{-1}}}{b_{\min{}}^2}\right)
    \\ &\leq \frac{4C_1^2}{b_{\min{}}}\left(   \|f\|_{\dotH^{-1+s}} + \frac{2}{b_{\min{}}}\|f\|_{\dotH^{-1}}\|b\|_{C^{0,s}} \right).
  \end{align*}
\end{proof}
\subsection{Bounds on exponentials and products of functions}
\begin{lemma}
  \label{lemma:Cka-exp}
  Let $k\in\N_0$, $\alpha\in [0,1]$, and $d\in\N$.
  There exists a constant $C$ that depends on
  $k$ and on $d$ such that for all $f\in C^{k, \alpha}(\bbT^d)$,
  \begin{equation*}
    \| e^f \|_{C^{k, \alpha}(\bbT^d)} \leq C e^{\|f\|_{L^{\infty}(\bbT^d)}} \| f\|_{C^{k, \alpha}(\bbT^d)}\left( \| f\|_{C^{k}(\bbT^d)} +1\right)^{k}.
  \end{equation*}
\end{lemma}
\begin{proof}
  We will use the Faà di Bruno formula
  \cite[eq.~(1.23)]{CoDaNi2010}, valid for any $0 \ne \beta \in \bbN^d_0$,
  \begin{equation*}
    (\partial^\beta e^f) (x) = e^{f(x)} \sum_{\ell=1}^{|\beta|} \sum_{\substack{\eta_1, \dots,\eta_\ell \in \bbN^d_0 \\\eta_1+\dots+\eta_\ell = \beta}}\gamma_{\eta_1, \dots,\eta_\ell} \prod_{i=1}^\ell \left(\partial^{\eta_i} f  \right)(x)\eqqcolon e^{f(x)}g_\beta(x),
  \end{equation*}
  where the constants $\gamma_{\eta_1, \dots, \eta_\ell}$ depend on $|\beta|$
  and on $d$. Start by remarking that
  \begin{equation*}
    (\partial^\beta e^f)(x) -
    (\partial^\beta e^f)(y) = \left( e^{f(x)} - e^{f(y)} \right)g_{\beta}(x) + e^{f(y)} \left( g_\beta(x) - g_\beta(y) \right) \eqqcolon (I) + (II).
  \end{equation*}
  For any multiindex $\beta$ there exists $C_{|\beta|, d}>0$ such that for all $x\in \bbT^d$
  \begin{equation*}
    |g_\beta(x)| \leq C_{|\beta|, d} \| f\|_{C^{|\beta|}(\bbT^d)}\left( \| f\|_{C^{|\beta|}(\bbT^d)}+1  \right)^{|\beta|-1}.
  \end{equation*}
  In addition, for all $x, y$
  \begin{equation}
    \label{eq:Cka-1}
    \begin{aligned}
    | e^{f(x)} - e^{f(y)} | &\leq
    e^{\max\{|f(x)|, |f(y)|\}} | f(x)  - f(y) |
    \\ &\leq
    e^{\|f\|_{L^\infty(\bbT^d)}} \|f\|_{C^{0, \alpha}(\bbT^d)}  |x-y|^\alpha,
  \end{aligned}
\end{equation}
  hence
  \begin{equation}
    \label{eq:Cka-2}
    |(I)| \leq C_{|\beta|, d} e^{\|f\|_{L^\infty(\bbT^d)}}\|f\|_{C^{|\beta|, \alpha}(\bbT^d)}\left( \|f\|_{C^{|\beta|}(\bbT^d)} +1\right)^{|\beta|} |x-y|^\alpha.
  \end{equation}
  We now estimate the term $(II)$.
Note that
  \begin{multline*}
    \prod_{i=1}^\ell (\partial^{\eta_i}f)(x)-
    \prod_{i=1}^\ell (\partial^{\eta_i}f)(y) \\ =
    \sum_{n=1}^\ell\left[ \bigg(\prod_{i=1}^{n-1}(\partial^{\eta_i}f)(x)   \bigg)\bigg( (\partial^{\eta_n}f )(x) - (\partial^{\eta_n}f )(y)\bigg)\bigg( \prod_{i=n+1}^\ell (\partial^{\eta_i}f)(y) \bigg)\right],
  \end{multline*}
  whence
  \begin{equation*}
    | g_\beta(x)  - g_\beta(y)|\leq
    \sum_{\ell=1}^{|\beta|} \sum_{\substack{\eta_1, \dots,\eta_\ell \in \bbN^d_0 \\\eta_1+\dots+\eta_\ell = \beta}}\gamma_{\eta_1, \dots,\eta_\ell} \ell \| f\|_{C^{|\beta|}(\bbT^d)}^{\ell-1} \|f\|_{C^{|\beta|,\alpha}(\bbT^d)}|x-y|^\alpha.
  \end{equation*}
  We conclude that there exists a constant $C_{|\beta|,d}>0$ such that, for all
  $x, y$
  \begin{equation}
    \label{eq:Cka-3}
    |(II)| \leq C_{|\beta|,d}e^{\|f\|_{L^\infty(\bbT^d)}} \|f\|_{C^{|\beta|, \alpha}(\bbT^d)}(\|f\|_{C^{|\beta|}(\bbT^d)}+1)^{|\beta|} |x-y|^\alpha.
  \end{equation}
  Combining \eqref{eq:Cka-1}, \eqref{eq:Cka-2}, and \eqref{eq:Cka-3} concludes
  the proof.
\end{proof}
\begin{lemma}
  \label{lemma:Cka-mult}
  Let $k\in\N_0$, $\alpha\in (0,1)$, and let $\Omega\subset\R^d$ be bounded. Then
  $C^{k,\alpha}(\Omega; \bbC)$ is closed under multiplication and in particular there
  exists $C>0$ that depends on $k$ such that for all $f,g\in
  C^{k,\alpha}(\Omega; \bbC)$,
  \begin{equation*}
    \| fg \|_{C^{k, \alpha}(\Omega)} \leq C\| f \|_{C^{k, \alpha}(\Omega)}   \| g \|_{C^{k, \alpha}(\Omega)}.
  \end{equation*}
\end{lemma}
\begin{proof}
  We have, for any $x, y\in\Omega$ and $\beta\in\N^d_0$ with $|\beta|=k$,
  \begin{multline*}
    \big| ( \partial^\beta  (fg))(x) -
    ( \partial^\beta  (fg))(y)\big|
    \\
    \begin{aligned}[t]
    & = \left| \sum_{\eta\leq \beta}\binom{\beta}{\eta}\left[ (\partial^\eta u)(x)(\partial^{\beta-\eta} v) (x) - (\partial^\eta u)(y)(\partial^{\beta-\eta}v)(y)  \right]\right|
      \\ & \leq
           \sum_{\eta\leq \beta}\binom{\beta}{\eta} \left| (\partial^\eta u)(x) - (\partial^\eta u)(y)\right| \left| (\partial^{\beta-\eta} v) (x) \right| + \left|(\partial^{\beta-\eta} v) (x)  - \partial^{\beta-\eta}v)(y)\right| \left| (\partial^\eta u)(y) \right|
           \\ & \leq
                \sum_{\eta\leq \beta}\binom{\beta}{\eta} 2|x-y|^\alpha
                \|v\|_{C^{k, \alpha}(\Omega)} \| u \|_{C^{k, \alpha}(\Omega)}.
    \end{aligned}
  \end{multline*}
\end{proof}
\begin{lemma}
  \label{lemma:Cka-expdiff}
  Let $k\in\N_0$, $\alpha\in [0,1)$, and $d\in\N$. There exists a constant $C$ that depends on
  $k$ and on $d$ such that for all $f,g\in C^{k, \alpha}(\bbT^d; \bbC)$,
  \begin{multline*}
    \| e^f - e^g \|_{C^{k, \alpha}(\bbT^d)} \\
    \leq C e^{\|f\|_{L^{\infty}(\bbT^d)} +  \|g\|_{L^{\infty}(\bbT^d)}}\left( \| f \|_{C^{k}(\bbT^d)}  + 1\right)^{k} \left( \| f - g\|_{C^{k}(\bbT^d)} + 1 \right)^k \| f \|_{C^{k, \alpha}(\bbT^d)}\| f - g\|_{C^{k, \alpha}(\bbT^d)}.
  \end{multline*}
\end{lemma}
\begin{proof}
    The statement follows from the equality $e^f - e^g = e^f(1-e^{g-f})$ and the
    application of Lemmas \ref{lemma:Cka-mult} and \ref{lemma:Cka-exp}.
    Note that the term $\| 1 - e^{g-f}\|_{L^\infty}$
    can be treated as in \eqref{eq:Cka-1}.
\end{proof}
\begin{lemma}
  \label{lemma:CH-ineq}
  Suppose that $k\in \N_0$ and either
  \begin{itemize}
    \item $s=0$ and $0\leq t < 1$, or
    \item $0< s<t<1$.
  \end{itemize}
    Then, there exists $C=C(k,s,t)>0$ such that for all $a\in C^{k, t}(\bbT^d; \bbC)$ and
  all $u\in H^{k+s}(\bbT^d; \bbC)$,
  \begin{equation*}
    \| ua \|_{H^{k+s}(\bbT^d)} \leq C \|a\|_{C^{k, t}(\bbT^d)} \|u\|_{H^{k+s}(\bbT^d)}.
  \end{equation*}
\end{lemma}
\begin{proof}
  The case $s=0$ follows from Leibniz rule. To prove the statement for $s>0$, it
  is therefore sufficient to consider the case $k=0$.
  We use the equivalence between the $H^s(\bbT^d)$ norm and the
 Sobolev-Slobodecki norm (see, e.g., \cite{Benyi2013}), which implies
 \begin{equation*}
   \| au\|^2_{H^s(\bbT^d)}\simeq\int_{\bbT^d} \int_{[-\frac{1}{2}, \frac{1}{2})^d} \frac{|( au )(x+h) - (au)(x)|^2}{|h|^{d+2s}}dh\,dx.
 \end{equation*}
 Then,
 \begin{align*}
   \| au\|^2_{H^s(\bbT^d)}
   & \begin{multlined}[t][.7\textwidth]
       \lesssim
       \int_{\bbT^d} \int_{[-\frac{1}{2}, \frac{1}{2})^d} \frac{| a (x+h) - a(x)|^2|u(x)|^2}{|h|^{d+2s}}dh\,dx
\\       +
       \int_{\bbT^d} \int_{[-\frac{1}{2}, \frac{1}{2})^d} \frac{|u(x+h) - u(x)|^2|a(x+h)|^2}{|h|^{d+2s}}dh\,dx
     \end{multlined}
   \\
   & \lesssim
     \|a\|_{C^{0, t}}^2\int_{\bbT^d} \int_{[-\frac{1}{2}, \frac{1}{2})^d}\frac{|u(x)|^2}{|h|^{d-2(t-s)}}dh\,dx + \| a\|_{L^\infty}^2 \| u\|^2_{H^s}
   \\ & \leq C(t-s)
        \|a\|_{C^{0, t}}^2 \| u \|_{L^2}^2+ \| a\|_{L^\infty}^2 \| u\|^2_{H^s}.
 \end{align*}
\end{proof}
\subsection{Proof of Proposition \ref{lemma:diffusion-assumption}}
\label{sec:diffusion-assumption-proof}
\begin{proof}[Proof of Proposition \ref{lemma:diffusion-assumption}]
  Consider $s>0$ with $s\notin\N$ that we will fix at the end of the proof. Let
  $k =\lfloor s \rfloor$, $\alpha = s-k+\delta_1$ with sufficiently small
  $\delta_1>0$ so that $\alpha<1$. Let then $a_1, a_2\in C^{k,
    \alpha}(\bbT^d; \bbC)$, and $u_i = \calG(a_i)$, $i=1,2$.   When $s\in (0,1)$ it follows from Proposition \ref{prop:Hs-poisson}
  that
  \begin{equation}
    \label{eq:H1s-ui}
    \| u_i\|_{\dotH^{1+s}} \leq Ce^{\|a_i\|_{L^\infty}}\left(  \|f\|_{\dotH^{-1+s}}+ e^{2\|a_i\|_{L^\infty}}\|f\|_{\dotH^{-1}}\|a_i\|_{C^{0, s}}\right).
  \end{equation}
  We now consider the case $s\geq 1$ and use the equality
  \begin{equation}
    \label{eq:laplace-decomp}
    -\Delta u_i = e^{-a_i} \left( \nabla e^{a_i}\cdot\nabla u_i + f \right).
  \end{equation}
  Then,
  \begin{align*}
    \| u_i\|_{\dotH^{s+1}}
    & \lesssim \| e^{-a_i}\left( \nabla e^{a_i}\cdot\nabla u_i + f \right) \|_{\dotH^{s-1}}
    \\
    \overset{\text{L.~\ref{lemma:CH-ineq}}}
    &{\lesssim}
      \| e^{-a_i} \|_{C^{k-1, \alpha}}
      \left( \| e^{a_i}\|_{C^{k, \alpha}} \|u_i\|_{\dotH^{s}} + \|f \|_{\dotH^{s-1}}\right) ,
      \\ \overset{\text{L.~\ref{lemma:Cka-exp}}}
    &{\lesssim}
      e^{\|a_i\|_{L^{\infty}}} \| a_i\|_{C^{k-1, \alpha}}\left( \| a_i\|_{C^{k-1}}+1 \right)^{k-1} \left(e^{\|a_i\|_{L^{\infty}}} \| a_i\|_{C^{k, \alpha}}\left( \| a_i\|_{C^{k}} +1\right)^{k}\|u_i\|_{\dotH^{s}} + \| f\|_{\dotH^{s-1}}  \right).
  \end{align*}
  We can conclude from the above and \eqref{eq:H1s-ui} that there exist $\beta_1, \beta_2>0$ such that, for all
  $s\geq 0$,
  \begin{equation}
    \label{eq:bound-ui}
    \| u_i \|_{\dotH^{s+1}} \lesssim e^{\beta_1 \|a_i\|_{L^\infty}} (\| a_i\|_{C^{k, \alpha}}+1)^{\beta_2 k^2} \|f\|_{\dotH^{s-1}}.
  \end{equation}
  Let now again $s\in(0,1)$. Since
  \begin{equation*}
    -\nabla\cdot\left[ e^{a_1}(\nabla u_1 - \nabla u_2) \right] = \nabla\cdot\left[ (e^{a_1} - e^{a_2})\nabla u_2 \right],
  \end{equation*}
from Proposition \ref{prop:Hs-poisson}  we have
  \begin{multline*}
    \| u_1 - u_2\|_{\dotH^{1+s}} \leq Ce^{\|a_1\|_{L^\infty}}\bigg(
    \|\nabla\cdot\left[ (e^{a_1} - e^{a_2})\nabla u_2\right]  \|_{\dotH^{-1+s}} \\ + e^{2\|a_1\|_{L^\infty}}\|\nabla\cdot\left[ (e^{a_1} - e^{a_2})\nabla u_2\right]  \|_{\dotH^{-1}}\|e^{a_1}\|_{C^{0, s}}\bigg).
  \end{multline*}
  We then estimate
  \begin{multline*}
    \|\nabla\cdot\left[ (e^{a_1} - e^{a_2})\nabla u_2\right]  \|_{\dotH^{-1+s}}
    \\
    \begin{aligned}
    &\leq
      \| (e^{a_1} - e^{a_2})\nabla u_2  \|_{\dotH^{s}}
    \\
    \overset{\text{L.~\ref{lemma:CH-ineq}}}
    &{\leq}
      \| e^{a_1} - e^{a_2}\|_{C^{0, \alpha}}\| u_2  \|_{\dotH^{1+s}}
    \\
    \overset{\text{L.~\ref{lemma:Cka-expdiff}}}
    &{\lesssim}
 e^{\|a_1\|_{L^{\infty}} +  \|a_2\|_{L^{\infty}}} \| a_1 \|_{C^{0, \alpha}}\| a_1 - a_2\|_{C^{0, \alpha}}\| u_2  \|_{\dotH^{1+s}}
    \\
    \overset{\text{\eqref{eq:H1s-ui}}}
    &{\lesssim}
e^{\|a_1\|_{L^{\infty}} +  2\|a_2\|_{L^{\infty}}} \| a_1 \|_{C^{0, \alpha}}\| a_1 - a_2\|_{C^{0, \alpha}}
\left(  \|f\|_{\dotH^{-1+s}}+ e^{2\|a_2\|_{L^\infty}}\|f\|_{\dotH^{-1}}\|a_2\|_{C^{0, s}}\right)
    \end{aligned}
  \end{multline*}
  and
  \begin{align*}
    \|\nabla\cdot\left[ (e^{a_1} - e^{a_2})\nabla u_2\right]  \|_{\dotH^{-1}}
    &\leq
    \| (e^{a_1} - e^{a_2})\nabla u_2  \|_{L^2}
    \\ & \leq
    \| e^{a_1} - e^{a_2}\|_{L^\infty}\|u_2  \|_{\dotH^1}
    \\ & \leq
    e^{2\|a_2\|_{L^\infty}}\| a_1 - a_2\|_{L^\infty} \| f\|_{\dotH^{-1}}.
  \end{align*}
  When $s\geq 1$, from \eqref{eq:laplace-decomp} it also follows that
  \begin{equation*}
    -\Delta(u_1 - u_2) = e^{-a_1} \left( \nabla (e^{a_1} - e^{a_2})\cdot \nabla u_1 + \nabla e^{a_2}\cdot( \nabla u_1 - \nabla u_2 )  \right) + (e^{-a_1} -  e^{-a_2})\left( \nabla e^{a_2}\cdot\nabla u_2   \right).
  \end{equation*}
  Using Lemmas \ref{lemma:Cka-mult} and \ref{lemma:Cka-expdiff}, we estimate,
  for $s\geq1$,
  \begin{multline*}
    \|\nabla (e^{a_1} - e^{a_2})\cdot \nabla u_1 \|_{\dotH^{s-1}}\\
    \leq C
    e^{\|a_1\|_{L^\infty} + \|a_2\|_{L^\infty}}\left( \|a_1\|_{C^k} + 1
    \right)^k\left( \|a_1-a_2\|_{C^k} + 1 \right)^k\|a_1\|_{C^{k,\alpha}} \|a_1 - a_2\|_{C^{k, \alpha}}\| u_1\|_{\dotH^{s}},
  \end{multline*}
  then
  \begin{equation*}
    \|\nabla e^{a_2}\cdot( \nabla u_1 - \nabla u_2 )  \|_{\dotH^{s-1}} \leq C e^{\|a_2\|_{L^\infty}} \left( \|a_2\|_{C^k}+1 \right)^k\|a_2\|_{C^{k, \alpha}} \| u_1 - u_2 \|_{\dotH^{s}},
  \end{equation*}
  and finally
  \begin{multline*}
    \|(e^{-a_1} -  e^{-a_2})\left( \nabla e^{a_2}\cdot\nabla u_2   \right)
    \|_{H^{s-1}}
    \leq C
    e^{\|a_1\|_{L^\infty} + 2\|a_2\|_{L^\infty}}\left( \|a_1\|_{C^{k-1}}  +
      1\right)^{k-1}\\ \times  \left( \|a_2\|_{C^k}+ 1\right)^{k-1}\left( \|a_1-a_2\|_{C^k} +1 \right)^{k-1}
    \|a_1\|_{C^{k-1,\alpha}} \|a_1 - a_2\|_{C^{k-1, \alpha}} \|a_2\|^{2}_{C^{k, \alpha}} \| u_2\|_{H^{s}}.
  \end{multline*}
  Iterating as above,
  we find that there exist $\beta_3, \beta_4>0$ such that
  \begin{multline*}
    \| u_1 - u_2 \|_{H^{s+1}} \\
    \leq e^{\beta_3 (\|a_1\|_{L^\infty}+ \|a_2\|_{L^\infty})} \left(
      \|a_1\|_{C^{k, \alpha}} + \|a_2\|_{C^{k, \alpha}} + \|a_1-a_2\|_{C^{k, \alpha}} +1\right)^{\beta_4 k^2} \| a_1 -a_2\|_{C^{k, \alpha}}.
  \end{multline*}
We have therefore shown that
\begin{equation*}
\norm[\dotH^{s+1}]{\calG(a)}\le C\exp(\alpha\norm[C^{k,\alpha}]{a}),
\end{equation*}
and
\begin{equation}
  \label{eq:cont-G-proof}
\norm[\dotH^{s+1}]{\calG(a_1)-\calG(a_2)}\le C \exp(\alpha(\norm[C^{k,\alpha}]{a_1}+\norm[C^{k,\alpha}]{a_2}))\norm[C^{k,\alpha}]{a_1-a_2}.
\end{equation}
 We now fix $s = r-\delta$ for
 sufficiently small $\delta>0$ so that $\dotH^{s +1} = \dotH^{r+1-\delta}$.
 If $r\notin\N$,
 \begin{equation*}
   \dotH^{r+d/2} \hookrightarrow C^{\lfloor r \rfloor, r-\lfloor r \rfloor} = C^{\lfloor s \rfloor, s-\lfloor s \rfloor+\delta} \eqqcolon C^{k, \alpha}.
 \end{equation*}
 If $r\in\N$, instead, we have
 \begin{equation*}
   \dotH^{r+d/2} \hookrightarrow C^{r-1,\gamma}, \qquad \forall \gamma\in (0,1),
 \end{equation*}
 hence
 \begin{equation*}
   \dotH^{r+d/2}\hookrightarrow C^{\lfloor s\rfloor, s-\lfloor s \rfloor+\delta_1}  \eqqcolon C^{k, \alpha}
 \end{equation*}
 as long as $\delta_1  < \delta$ and therefore $s-\lfloor s \rfloor +\delta_1 =
 1-\delta+\delta_1 \in (0,1)$.
 This concludes the proof of the bounds in Assumption \ref{ass:G}.
 To show the holomorphy of $\calG$
 from a strip in $\calX + i\calX$ to  $\calY_\bbC$,
 we start by remarking that, for two real-valued functions $a, b\in \calX$,
 \begin{equation*}
   \Re(e^{a+ib}) =  e^a\cos(b).
 \end{equation*}
 Hence, if $\|b\|_{L^\infty(\Omega)}<\pi/2$, then $\Re(e^{a +ib}(x))>0$ for
 almost all $x\in \Omega$.
 By the continuous embedding of $\calX $ into
 $L^\infty(\Omega)$, there exists therefore $\theta$ such that $\Re(e^z)>0$ for
 all $z$ in the strip $\calS^\calX_\theta$. The operator $\tilde\calG:c\mapsto
 \tilde u$, where $\tilde u$ is solution to $-\nabla\cdot(c\nabla\tilde u) = f$ in
 $\bbT^d$ is holomorphic from $\{c\in \calX_{\bbC}: \Re (c(x))\geq c_0>0 \text{ a.e.}\} \to \dotH^1_{\bbC}$.
 By the holomorphy of the exponential,
 we conclude that the map $\calG$ is holomorphic from
 $\calS^\calX_\theta \to \dotH^1_{\bbC}$.
We have shown in \eqref{eq:cont-G-proof} that $\calG$ is continuous from
$\calS^\calX_\theta$ to $\calY_\bbC = \dotH^{r+1-\delta}_\bbC\hookrightarrow \dotH^1_\bbC$.
It follows from \cite[Proposition 1]{HSZ22_1015} (see also Remark
\ref{remark:holomorphy-embedding}) that
$\calG : \calS^\calX_\theta\to \calY_\bbC$ is holomorphic.
\end{proof}
\bibliographystyle{amsalpha}
\bibliography{mybibfile}
\end{document}